%% file: manuscript.tex
\documentclass[a4paper,12pt]{amsart}

\usepackage[utf8]{inputenc}
\usepackage[T1]{fontenc}
\usepackage[english]{babel}

\usepackage[margin=7em]{geometry}

\usepackage{amssymb, amsfonts}

\usepackage{bbm}

\usepackage[all]{xy}

\usepackage{enumerate}

\usepackage{booktabs}

\usepackage{listings}

\usepackage{url}

\numberwithin{equation}{section}

\newtheorem{theoremcounter}{theoremcounter}[section]

\newtheorem{algorithm}[theoremcounter]{Algorithm}

\newtheorem{corollary}[theoremcounter]{Corollary}
\newtheorem{definition}[theoremcounter]{Definition}
\newtheorem{example}[theoremcounter]{Example}
\newtheorem{lemma}[theoremcounter]{Lemma}
\newtheorem{proposition}[theoremcounter]{Proposition}
\newtheorem{remark}[theoremcounter]{Remark}
\newtheorem{remarks}[theoremcounter]{Remarks}
\newtheorem{theorem}[theoremcounter]{Theorem}

\input{shortcuts}

\begin{document}

\title{Computing Genus $1$ Jacobi Forms}

\author{Martin Raum}
\address{ETH, Dept. Mathematics, Rämistraße 101, CH-8092, Zürich, Switzerland}
\email{martin.raum@math.ethz.ch}
\urladdr{http://www.raum-brothers.eu/martin/}
\thanks{The author is supported by the ETH Zurich Postdoctoral Fellowship Program and by the Marie Curie Actions for People COFUND Program.}


\input{abstract}
\maketitle


\input{introduction}
\vspace{0.3em}
{\tit Acknowledgment: The author thanks Max Plank Institute for Mathematics, Bonn, Germany for granting access to its computation server.  The referee helped improving the quality of this paper by useful comments also on the accompanying implementation.}

\input{discrimiantforms}
\input{modularforms}
\input{algorithm}
\input{heckeoperators}

\input{newformsalgorithm}
\input{implementation}
\input{data}
\input{specialdivisors}

\addtocontents{toc}{\protect\setcounter{tocdepth}{0}}
\bibliographystyle{amsalpha}
\bibliography{bibliography}

\end{document}

%% file: shortcuts.tex
\newcommand{\tit}{\itshape}

\newcommand{\nbd}{\nobreakdash-\hspace{0pt}}

\renewcommand{\frak}{\ensuremath{\mathfrak}}
\newcommand{\cal}{\ensuremath{\mathcal}}

\newcommand{\frakc}{\ensuremath{\frak{c}}}

\newcommand{\frake}{\ensuremath{\frak{e}}}

\newcommand{\cC}{\ensuremath{\cal{C}}}
\newcommand{\cD}{\ensuremath{\cal{D}}}
\newcommand{\cE}{\ensuremath{\cal{E}}}
\newcommand{\cF}{\ensuremath{\cal{F}}}

\newcommand{\cH}{\ensuremath{\cal{H}}}
\newcommand{\cI}{\ensuremath{\cal{I}}}
\newcommand{\cJ}{\ensuremath{\cal{J}}}

\newcommand{\cL}{\ensuremath{\cal{L}}}

\newcommand{\cR}{\ensuremath{\cal{R}}}
\newcommand{\cS}{\ensuremath{\cal{S}}}

\newcommand{\cU}{\ensuremath{\cal{U}}}

\newcommand{\rmG}{\ensuremath{\mathrm{G}}}

\newcommand{\rmJ}{\ensuremath{\mathrm{J}}}

\newcommand{\rmM}{\ensuremath{\mathrm{M}}}

\newcommand{\rmd}{\ensuremath{\mathrm{d}}}

\newcommand{\amid}{\ensuremath{\mathop{\mid}}}

\newcommand{\ZZ}{\ensuremath{\mathbb{Z}}}
\newcommand{\QQ}{\ensuremath{\mathbb{Q}}}
\newcommand{\RR}{\ensuremath{\mathbb{R}}}
\newcommand{\CC}{\ensuremath{\mathbb{C}}}

\newcommand{\isdiv}{\amid}
\newcommand{\nisdiv}{\ensuremath{\mathop{\nmid}}}

\newcommand{\Mat}[2]{\ensuremath{\mathrm{M}_{#1}(#2)}}

\newcommand{\GL}[1]{\ensuremath{\mathrm{GL}_{#1}}}
\newcommand{\SL}[1]{\ensuremath{\mathrm{SL}_{#1}}}
\newcommand{\Sp}[1]{\ensuremath{\mathrm{Sp}_{#1}}}
\newcommand{\GSp}[1]{\ensuremath{\mathrm{GSp}_{#1}}}
\newcommand{\Orth}[1]{\ensuremath{\mathrm{O}_{#1}}}

\newcommand{\tr}{\ensuremath{\mathrm{tr}}}

\newcommand{\diag}{\ensuremath{\mathrm{diag}}}

\newcommand{\slashdiv}{\ensuremath{\mathop{/}}}

\newcommand{\lspan}{\ensuremath{\mathop{\mathrm{span}}}}

\newcommand{\HS}{\mathbb{H}}

\newcommand{\td}{\tilde}

\newcommand{\ov}{\overline}

\newcommand{\Mp}[1]{\ensuremath{\mathrm{Mp}_{#1}}}

\newcommand{\disc}{\ensuremath{\mathrm{disc}}}

\renewcommand{\pmod}[1]{\ensuremath{\;(\mathrm{mod}\, #1)}}
\newcommand{\eqPic}{\mathop{\cong_\mathrm{Pic}}\,}

%% file: abstract.tex
\begin{abstract}
We develop an algorithm to compute Fourier expansions of vector valued modular forms for Weil representations.  As an application, we compute explicit linear equivalences of special divisors on modular varieties of orthogonal type.  We define three families of Hecke operators for Jacobi forms, and analyze the induced action on vector valued modular forms.  The newspaces attached to one of these families are used to give a more memory efficient version of our algorithm.
\end{abstract}

%% file: introduction.tex
\section{Introduction}
\label{sec:introduction}

Let $\cL$ be a lattice of signature $(2, n)$.  One can define a modular variety~$X_{\cL}$ attached to $\cL$ (see Section~\ref{sec:specialdivisors}).  The Picard group ${\rm Pic}(X_\cL)$ of $X_\cL$ contains a subgroup ${\rm Pic}_{\rm sp}(X_\cL)$ that is spanned by so-called special divisors $Z(m, \mu)$.  In the past decade, we learned how to compute the rank of ${\rm Pic}_{\rm sp}(X_\cL)$ under certain assumptions on $\cL$.  Most recently, this problem was resolved by Bruinier~\cite{Br12} (see~\cite{Br02} for earlier results) in the case $n > 2$ assuming that $\cL$ splits a hyperbolic plane $\cU = \left(\begin{smallmatrix} 0 & 1 \\ 1 & 0 \end{smallmatrix}\right)$ and a rescaled copy $\cU(N)$ of $\cU$.  In this setting, the rank of ${\rm Pic}_{\rm sp}(X_\cL)$ essentially equals the dimension of a certain space of vector valued modular forms.  In this paper, we develop a method to compute relations of special divisors in ${\rm Pic}_{\rm sp}(X_\cL)$.
\begin{theorem}
\label{thm:algorithm_existence}
There exists an algorithm, which, for any $\cL = \cU \oplus \cU(N) \oplus \cL'(-1)$ with positive definite $\cL'$, computes linear equivalences between any finite set of special divisors $Z(m, \mu)$.
\end{theorem}

For example, for $N = 1$ and $\cL' = (8)$, a lattice of rank~$1$, we have
\begin{gather*}
  Z(\tfrac{17}{16}, 3)
\eqPic
  2 \, Z(\tfrac{9}{16}, 1) - 2 \, Z(\tfrac{1}{16}, 3)
\text{.}
\end{gather*}
More such relations will be given in Section~\ref{sec:specialdivisors}.

\subsection{Special Picard groups and vector valued modular forms}
\label{ssec:introduction:specialpicardgroups}

The computation of Picard groups is intricate, and explicit determination of linear equivalences of divisors is even more difficult.  In the case of orthogonal varieties Borcherds provides a construction of meromorphic functions, called Borcherds products, whose divisors are explicit sums of special divisors~\cite{Bo98}.  If $n \ge 2$, there is hope that the converse of Borcherds result holds.  Every meromorphic modular form whose divisor is a linear combination of special divisors should be a Borcherds product.  As we have mentioned, in some cases this converse theorem is proved.

Borcherds construction works as follows.  Given a weakly holomorphic vector valued modular form of type $\rho_\cL$ (the Weil representation of the discriminant form associated to $\cL$) consider its Fourier expansion.
\begin{gather*}
  \sum_{\mu} \sum_{m \in \QQ} a(m, \mu) \exp(2 \pi i\, m \tau) \frake_\mu
\text{.}
\end{gather*}
Here the components are index by $\mu$ and the $\frake_\mu$ are certain basis elements for the representation space of $\rho_\cL$.  If $a(m, \mu) \in \ZZ$ for $m < 0$ and $a(0, 0) = 0$, then there is a meromorphic function on $X_\cL$ whose divisor equals
\begin{gather*}
  \sum_\mu \sum_{0 > m \in \QQ} a(m, \mu) Z(m, \mu)
\text{.}
\end{gather*}
In the light of the converse theorem, we are thus reduced to considering principal parts of weakly holomorphic modular forms, when studying ${\rm Pic}_{\rm sp}(X_\cL)$.

Serre duality can be used to determine which principal parts occur~\cite{Fi87, Bo00b}.  This brings into focus holomorphic vector valued modular forms of weight~$k = \frac{2 + n}{2}$ and type~${\check \rho}_{\cL} = \rho_{\cL(-1)}$ (the dual of $\rho_\cL$), the space of which we will denote by $\rmM_k({\check \rho}_\cL)$.  Since $\cL$ is indefinite these spaces are difficult to compute.  In the literature, one can find two methods which allow to determine Fourier expansions of elements of $\rmM_k({\check \rho}_\cL)$.  First of all, there are Eisenstein series, whose coefficients can be calculated efficiently~\cite{BK01}.  If $\cL$ splits two hyperbolic planes, similarly efficient calculations can be done for theta series~\cite{Si51}(\cite{Bo98} gives explanations in terms of vector valued theta series).  The drawback of relying on these constructions is that $\rmM_k({\check \rho}_\cL)$ is often not spanned by Eisenstein series and theta series.  A second method, which was brought to attention by Scheithauer~\cite{Sch04, Sch08}, is to average usual modular forms.  One can show that every vector valued modular form arises this way.  The drawback in this case is that one has to explicitly know Fourier expansions at all cusps of modular forms for groups of non-squarefree level.  This is an issue which is not easy to settle, and which, for example, in~\cite{Ra12} was dealt with by an adhoc approach.  In general, averaging does not seem well suited for large scale computations.  As a side node, we mention that there is a third method, that apparently is not treated in the literature.  One can use Eisenstein series and theta series for harmonic polynomials in order to construct elements of $M_{k + l 12}({\check \rho}_{\cL})$.  Suitable linear combinations will have vanishing order greater than or equal to $l$, and this vanishing order is sufficient to obtain elements of $\rmM_k ({\check \rho}_\cL)$ after dividing by $\Delta^l$, where $\Delta$ is the discriminant modular form.

Summarizing, we can say that averaging as described by Scheithauer is the only known systematic way to span $\rmM_k ({\check \rho}_\cL)$.  This method, however, relies on another difficult computational problem.  We therefore introduce Jacobi forms into this picture.  They will allow us to compute $\rmM_k ({\check \rho}_\cL)$ in a reliable, easy to implement way.

\subsection{Vector valued modular forms and Jacobi forms}

If $\cL$ was positive definite, for every $0 \ge k \in \frac{1}{2}\ZZ$ there would be a well-know link between $\rmM_{k} ({\check \rho}_\cL)$ and $\rmJ_{k + \frac{2 + n}{2}, \cL}$, the space of Jacobi forms of weight $k + \frac{2 + n}{2}$ and index~$\cL$ (see~\cite{Zi89} and Section~\ref{sec:modularforms} for definitions and details).  However, $\cL$ is never positive definite.  Section~\ref{sec:discriminantforms} contains the proof of the following theorem.
\begin{theorem}
\label{thm:positive_discriminant_forms}
There is an algorithm (Algorithm~\ref{alg:positive_discriminant_forms}) that, for any even lattice~$\cL$, computes an even, positive definite lattice~$\cL'$ satisfying $\cL \oplus \cU_1 \cong \cL' \oplus \cU_2$ for suitable unimodular lattices $\cU_1$ and $\cU_2$.
\end{theorem}
Note that with $\cL$ and $\cL'$ as in the previous theorem, we have $\rho_{\cL} \cong \rho_{\cL'}$.  After applying the theorem we are therefore reduced to the case of vector valued modular forms associated to positive definite lattices.  For the rest of this section, \emph{we let $\cL$ be a positive definite, even lattice of rank~$N$}.

The theta decomposition for Jacobi forms, which is revisited in Section~\ref{sec:modularforms}, allows us to consider $\rmJ_{k, \cL}$, $k \ge \frac{N}{2}$ instead of vector valued modular forms.  Note that $\rmJ_{k, \cL}$ is zero except if $k$ is integral.  It is useful to distinguish two types of Jacobi forms.  Jacobi forms of scalar index, that is those whose Jacobi index is a lattices of rank~$1$, are by now classical.  They were studied systematically in~\cite{EZ85}, but first instances of such forms already appear in Jacobi's work.  In Skoruppa's thesis~\cite{Sk84}, an algorithm to compute such Jacobi forms was given.  Together with the newform theory, hinted at in~\cite{EZ85} and proved in~\cite{SZ88, SZ89}, this provides us with an efficient means to calculate Fourier expansions of Jacobi forms, even if their weight and index are large.

The case of lattices with $N > 1$ was, so far, not considered systematically.  Applications that have appeared in literature rely on coincidences.  For example, some spaces of Jacobi forms can be spanned by products of theta series -- see~\cite{CG08} for one paper were this was made use of.  In Section~\ref{sec:algorithm}, we discuss an algorithm which gives us Fourier expansions of a basis of $\rmJ_{k, \cL}$ up to arbitrary precisions.  The idea behind this algorithm is as follows.  Restricting a Jacobi form $\phi(\tau, z)$ to $z = s^\tr z'$, $z' \in \CC$, $s \in \ZZ^N$ we obtain a Jacobi form of the same weight and scalar index.  We first prove that sufficiently many of these restrictions are enough to uniquely determine $\phi$.  In a very special case and for some fixed weights and indices, a similar result was used by Das~\cite{Das10} to analyze the structure of so-called hermitian Jacobi forms.

As a second step, we prove that one can construct Jacobi forms in this way.  Note that there is an analog of the restriction $z = s^\tr z'$ for formal Fourier expansions.  Suppose that $\phi$ is a formal Fourier expansion that has all symmetries inherent to Fourier expansions of Jacobi forms of weight~$k$ and index~$\cL$.  If sufficiently many restrictions of $\phi$ are Fourier expansions of Jacobi forms of correct scalar index, then $\phi$ is the Fourier expansion of a Jacobi form.  Algorithm~\ref{alg:jacobiforms} makes us of this fact to successively compute restrictions, until the only formal Fourier expansions that restrict to Jacobi forms are those which come from $\rmJ_{k, \cL}$.

\begin{theorem}
There is an algorithm (Algorithm~\ref{alg:jacobiforms}) that, for any $k \in \ZZ$ and even, positive definite lattice~$\cL$, computes Fourier expansions of Jacobi forms of weight~$k$ and index~$\cL$.
\end{theorem}

We remark that our algorithm depends on enumerating short vectors in $\cL$.  In order to use our algorithm we only need to enumerate the lattice once for all weights, as long as certain conditions are met.  Actually, $\cL$ has rank $6$ or $7$ for all computations that we performed.  These ranks are feasible even with non-parallel implementations of lattice enumeration that are not optimal.  In this sense, our algorithm is good enough to yield new results even without consideration of technical details of the implementation.  In more complicated cases, $\cL$ can be too large to enumerate sufficiently many vectors so that Algorithm~\ref{alg:jacobiforms} finds enough restrictions.  At the end of Section~\ref{sec:implementation} we comment on further developments that, we think, can mitigate the impact of large lattice ranks on performance.

An idea which similar to the one behind our algorithm was employed by Poor and Yuen in order to compute Fourier expansions of Siegel modular forms of genus~$4$~\cite{PY07}.  Restricting Siegel modular forms to many different elliptic modular curves, embedded into the Siegel upper half space, they obtained relations between Fourier coefficients.  When increasing the number of such restrictions these relations become sufficient to determine the space of Fourier expansions that come from Siegel modular forms inside the space of all possible Fourier coefficients.  Poor and Yuen could not prove that this is always the case, and such a proof is not available until today.  Their algorithm, however, works perfectly fine in practice.

There is one fundamental difference between our algorithm and the one in~\cite{PY07}.  While Poor and Yuen would have to increase the number of restrictions if they wanted to compute Fourier expansions up to higher precisions, this is not the case for Jacobi forms.  The reason for this is that~\eqref{eq:higherjacobiforms_fourier_relation1} and~\eqref{eq:higherjacobiforms_fourier_relation2} hold, while there is no correspondingly powerful tool in the case of Siegel modular forms.

\subsection{Hecke operators and newform theories}

Hecke theory and the resulting theory of newspace can be used to speed up computations of classical elliptic modular forms.  Section~\ref{sec:heckeoperators} contains investigations related to this problem.  In~\cite{SZ89}, it was proved that we have a direct sum decomposition
\begin{gather*}
  J_{k, (2 m)}
=
  \bigoplus_{l' l^2 \isdiv m} J_{k, (2 m \slashdiv l' l^2)}^{\rm new} \big| V_{l'} U_l 
\end{gather*}
of the space of Jacobi forms of index $\cL = (2 m)$.  This decomposition was already used in~\cite{Sk84} to provide an optimization for the computation of Jacobi forms of scalar index.  Generalizations of~$U_l$ and $V_l$ to Jacobi forms of arbitrary indices were provided in~\cite{Gr94}.  One might hope to find a decomposition analogue to the above for all Jacobi indices $\cL$, but this is not the case.  The lattice $\left(\begin{smallmatrix} 8 & 4 \\ 4 & 8 \end{smallmatrix}\right)$ provides us with a counter example in the following way.  Assuming that a direct sum decomposition into newspaces exists, it is not hard to write down a recursive formula for the dimensions of newspaces.  Computing these hypothetical dimensions of newspaces, one expects to obtain non-negative values, but if $k = 4$ one obtains $-1$.  We will consider the question which decomposition yields a direct sum in a sequel.

We investigate the induced action of both generalizations $U_s$ and $V_l$ on vector valued modular forms.  The action coming from $U_s$ coincides with the action of an operator that was already studied by Scheithauer and Bruinier~\cite{Sch11, Br12}.  However, the induced action of $V_l$ is new.  It corresponds to a curious extension of the discriminant group, and deserves further investigation.  As it is only loosely related to the theme of this paper, we defer this study to a future publication.

The main result of Section~\ref{sec:newforms_algorithm} says that newforms with respect to the Hecke operators $U_s$ can be characterized by restrictions to Jacobi forms of scalar index.  As a result, we are able to describe an Algorithm which makes use of the action of $U_s$ to reduce memory consumption of computations of Jacobi forms.  Since at this stage we are not able to prove a newspace decomposition for Jacobi forms which is a direct sum, this does not result in a reduction of runtime.  It does, however, enable us to split up computations.  This is a useful feature when computing large spaces of Jacobi forms.

\subsection{A reference implementation and two families of examples}

We provide an implementation of Algorithm~\ref{alg:jacobiforms} in Sage~\cite{sageticket-jacobi}.  Section~\ref{sec:implementation} contains instructions on how to install and use it.  In Section~\ref{sec:implementation}, we also discuss shortcomings of our implementation and possible improvements.  Regardless of these imperfections, our implementation can be used for basic computations.  It is also meant to be a reference implementation, which clarifies Section~\ref{sec:algorithm}.

The reference implementation allows to compute Fourier expansions of Jacobi forms of arbitrary weights $> 2 + \frac{N}{2}$.  We illustrate how to invoke the corresponding commands and give examples.  Note that Fourier expansions of any other weight and of weakly holomorphic Jacobi forms can be obtained by division by powers of the modular discriminant.

In Section~\ref{sec:data}, we provide Fourier expansions of vector valued elliptic modular forms for cyclic discriminant forms of even order.  If these discriminant forms come from positive definite rank~$1$ lattices, the attached vector valued modular forms can be easily computed using Jacobi forms of scalar index.  If the associated lattice of rank~$1$ is negative definite, then they correspond to skew-holomorphic Jacobi forms, whose Fourier expansions are not easily accessible.  In this case, only our new algorithm allows to perform explicit computations.  Section~\ref{sec:data} contains some printed tables of Fourier expansions.  Fourier expansions for weights between $\frac{5}{2}$ and $\frac{125}{2}$ can be found at the author's homepage~\cite{raumhomepage}.

\subsection{Further Applications}

The application to special Picard groups that was explained at the beginning of the introduction can be found in Section~\ref{sec:specialdivisors}.  At this place, we briefly discuss two further applications, that we will not treat in detail.  The first concerns explicit computations of Borcherds products, that was treated in~\cite{GKR11}.  We have already explained Borcherds's construction in Section~\ref{ssec:introduction:specialpicardgroups}.  This construction does not only work for meromorphic forms, but in general $a(0, 0)$ may be an arbitrary integral number.  A Borcherds product determined by a vector valued modular form with coefficients~$a(m,\mu)$ will have weight $a(0, 0) \slashdiv 2$, and it will be holomorphic if $a(m, \mu) \ge 0$ for all $m < 0$.  For our purposes, we restrict to the case of half\nbd  integral weights.  Such Borcherds products are important in several areas, including string theory~\cite{JS05} and $M_{24}$-moonshine~\cite{EOT10, CD12}.

Recall that we assume that $\cL$ has signature $(2, n)$, $n > 2$.  Borcherds products arise from weakly holomorphic vector valued modular forms of weight $\frac{2 - n}{2}$ and type $\rho_{\cL}$.  If, for example, $\cL$ is a lattice of the form $\cU \oplus \cU \oplus \cL'(-1)$, then $\rho_{\cL} \cong {\check \rho}_{\cL'}$.  Since $\cL'$ is necessarily positive definite, we can apply the theta decomposition without problems.  In order to obtain weakly holomorphic Jacobi forms, which serve as input data to Borcherds's multiplicative lift,  one can compute vector valued modular forms of weight $\frac{n - 2}{2} + 12 l$, for some $0 < l \in \ZZ$.  Dividing the resulting Fourier expansion by $\Delta^l$, where $\Delta$ is the discriminant modular form, gives us a basis of weakly holomorphic vector valued modular forms with poles at infinity of order less than or equal to~$l$.  The needed computations of Fourier expansion of Jacobi forms of index~$\cL'$ are rather easy for, say, $ n \le 4$.

The second application of vector valued modular forms which we want to discuss is based on reversing the ideas in~\cite{Sch08}.  Scheithauer pursued a detailed study of averages of scalar valued elliptic modular forms.  The Fourier expansions of resulting vector valued modular forms depend on the Fourier expansions of the input data at all cusps.  Modular forms that one starts with typically do not have squarefree level.  As a consequence, it is quite demanding to obtain expansions at cusps different from~$\infty$.  On the other hand, $\SL{2}(\ZZ)$ acts transitively on the set of cusps~$\cC$ of a congruence subgroup $\Gamma \subseteq \SL{2}(\ZZ)$.  Denote the associated permutation representation by $\rho$.  A modular form~$f$ of weight~$k$ for $\Gamma$ gives rise to a vector valued modular form of weight~$k$ and type~$\rho$.  Its components with respect to the obvious basis of $\CC[\cC]$ are simply given by the expansions~$f_\frakc$ of $f$ at the various elements $\frakc$ of $\cC$.

Since $\rho$ is unitary, it can be decomposed into a direct sum of Weyl representations.  Our work thus allows us to compute the Fourier expansion of the vector valued modular form with prescribed $f_\infty$.  Thus we obtain all $f_\frakc$.


%% file: discrimiantforms.tex
\section{Discriminant forms and positive definite lattices}
\label{sec:discriminantforms}

In this section, we relate discriminant forms and positive definite lattices.  In particular, we explain Algorithm~\ref{alg:positive_discriminant_forms}, which allows us to find a positive definite lattice that represents a given discriminant form.  A useful overview over discriminant forms can be found in~\cite{Ni79}.
\begin{definition}
Let~$D$ be a finite abelian group that is equipped with a quadratic form $D \rightarrow \QQ \slashdiv \ZZ,\, \gamma \mapsto \gamma^2$.  We call $D$ a discriminant form if the induced bilinear form $(\gamma, \delta) := (\gamma + \delta)^2 - \gamma^2 - \delta^2 \pmod{1}$ is non-degenerate.
\end{definition}
In~\cite{Wa63}, it was shown that all discriminant forms arise from lattices.  Even more was shown:  Two even, non-degenerate lattices are stably isomorphic if and only if the corresponding discriminant forms are isomorphic.  A special case of Theorem~1.10.1 in~\cite{Ni79}, tells us that, for every discriminant form, one can find a positive definite lattice which represents it.  The goal of this section is to provide an algorithm which computes one such lattice.

For later use, we fix some notation.  Let $\cL \subset \ZZ^N$ be a non-degenerate, even lattice, and write $\langle \cdot\,,\,\cdot\rangle_{\cL}$ for the induced even bilinear form.  Set $q_{\cL}(v) = \langle v, v \rangle_{\cL} \slashdiv 2$.  The dual lattice $\cL^\#$ consists of all vectors $v \in \cL \otimes \QQ$ such that $\langle v, w \rangle_{\cL} \in \ZZ$ for all $w \in \cL$.  Then the discriminant form~$\disc\, \cL$ attached to $\cL$ is defined as $\cL^\# \slashdiv \cL$ equipped with the quadratic form $q_\cL \pmod{1}$.  For many purposes, in particular in the context of Jacobi forms, it is good to identify $\cL^\#$ with $\ZZ^N$.  Given a Gram matrix $L$ of $\cL$, we find that $\cL$ then corresponds to $L \ZZ^N$.  To relieve notation, we normally write $\cL = L$ for a matrix $L$ when we mean the lattice spanned by the columns of $L$.  The quadratic form $q_\cL$ on $\cL^\#$ and $\cL$ is given by $v \mapsto L^{-1}[v] = v^\tr L^{-1} v$.

We start with a basic lemma on representation numbers of the unimodular lattice~$E_8$.
\begin{lemma}
\label{la:E8_represents_all_numbers}
Given any positive, even number $2 n$, there are vectors $v, w \in E_8$ of length $2n$ and scalar product $\langle v, w \rangle_{E_8} = 2n - 1$.
\end{lemma}
\begin{proof}
In order to prove the lemma, it suffices to show that the $\left(\begin{smallmatrix}2n & 2n -1 \\ 2n - 1 & 2n\end{smallmatrix}\right)$-th coefficient of the theta series attached to $E_8$ is non-zero.  By~\cite{Fr91}, this theta series is a Siegel modular form of degree~$2$ for the full modular group.  By results of Igusa~\cite{Ig84}, it is a (non-zero) multiple of the Siegel Eisenstein series $E^{(2)}_4$.  The Fourier coefficients of $E^{(2)}_4$ are invariant under unimodular transformations, so that it suffices to analyse the Fourier coefficients of
\begin{gather*}
  \begin{pmatrix}-1 & 0 \\ -1 & 1 \end{pmatrix}
  \begin{pmatrix}2n & 2n -1 \\ 2n - 1 & 2n\end{pmatrix}
  \begin{pmatrix}-1 & -1 \\ 0 & 1 \end{pmatrix}
=
  \begin{pmatrix}2n & 1 \\ 1 & 2\end{pmatrix}
\text{.}
\end{gather*}
That is, we have to analyze the first Fourier Jacobi coefficient of $E^{(2)}_4$.  It is a weight~$4$, index~$1$ Jacobi form in the sense of~\cite{EZ85}.  Since $E^{(2)}_4$ is a Maa\ss\ lift, this Fourier Jacobi coefficient is non-zero.  And since the space of weight~$4$, index~$1$ Jacobi forms is generated by the Jacobi Eisenstein series $E_{4, 1}$, we are reduced to analyzing the Fourier coefficients of this Jacobi form.  Its coefficients were studied in \cite{EZ85}, Theorem 2.1.  In particular, there is a remark preceding the theorem saying that the Fourier coefficients of $E_{4, 1}$ are positive.  This proves the lemma.
\end{proof}

The main result of this section is the following algorithm and Theorem~\ref{thm:positive_discriminant_forms}.
\begin{algorithm}
\label{alg:positive_discriminant_forms}
Given a lattice $\cL$ the following algorithm computes a positive definite lattice $\cL'$ that is stabily equivalent to $\cL$.
\begin{enumerate}[(1)]
\item Set $\cL' \leftarrow \cL$.  As long as $\cL'$ is not positive definite iterate the following steps.
\vspace{0.5em}

\item \label{it:alg:positive_discriminant_forms:choose_vector}
      Choose a vector $v \in \cL'$ of negative length which is as short as possible.  Set $n \leftarrow - q_{\cL'}(v)$.
\vspace{0.5em}

\item \label{it:alg:positive_discriminant_forms:E8_vectors}
      Choose a pair of vectors $w_1, w_2 \in E_8$ satisfying $q_{E_8}(w_1) = q_{E_8}(w_2) = n$ and $\langle w_1, w_2 \rangle_{E_8} = 2n - 1$.
\vspace{0.5em}

\item \label{it:alg:positive_discriminant_forms:new_lattice}
      Set $\cL \leftarrow \lspan( v + w_1, v + w_2 )^\perp \subset \cL' \oplus E_8$.
\end{enumerate}
\end{algorithm}
This proves Theorem~\ref{thm:positive_discriminant_forms}.  The algorithm also effectively reproves a theorem by Nikulin.
\begin{corollary}[Special case of Theorem~1.10.1 in~\cite{Ni79}]
\label{cor:positive_discriminant_forms}
Given a lattice~$\cL$ there is a positive definite lattice $\cL'$ that is stably isomorphic to~$\cL$.
\end{corollary}

\begin{proof}[Proof of Theorem~\ref{thm:positive_discriminant_forms}]
By what we have observed above, we have $D \cong \disc\, \cL'$ for some even, non-degenerate lattice~$\cL'$.  Applying Algorithm~\ref{alg:positive_discriminant_forms} we get the desired lattice $\cL$.
\end{proof}


\begin{proof}[Proof of correctness and termination of Algorithm~\ref{alg:positive_discriminant_forms}]
Whenever Step~\ref{it:alg:positive_discriminant_forms:choose_vector} is done, $\cL'$ is not positive definite.  In particular, there is some vector of negative length in $\cL'$.

By Lemma~\ref{la:E8_represents_all_numbers}, vectors $w_1$ and $w_2$ as in Step~\ref{it:alg:positive_discriminant_forms:E8_vectors} always exist.  The lattice $\lspan( v + w_1, v + w_2 )$ is a hyperbolic plane.  Indeed, we have $q_{\cL' \oplus E_8}(v + w_1) = q_{\cL'}(v) + q_{E_8}(w_1) = 1$, and $q_{\cL' \oplus E_8}(v + w_2) = 0$.  Further, $\langle v + w_1, v + w_2 \rangle_{\cL' \oplus E_8} = 2 q_{\cL'}(v) + \langle w_1, w_2 \rangle_{E_8} = -1$.

From this we see that the number of negative eigenvalues of $\lspan( v + w_1, v + w_2 )^\perp \subset \cL' \oplus E_8$ is one less than the number of negative eigenvalues of $\cL'$.  This proves that Algorithm~\ref{alg:positive_discriminant_forms} terminates and that the output $\cL'$ is positive definite.
\end{proof}

One should remark that the algorithm is not optimal in the sense that $\cL'$ possibly contains a unimodular lattice.
\begin{example}
The lattice $\cL = -\left(\begin{smallmatrix}2 & 1 \\ 1 & 2 \end{smallmatrix}\right)$ and the positive definite lattice
\begin{gather*}
  \cL'
=
\begin{pmatrix}
2 & 1 & 1 & 1 & -1 & -1 \\
1 & 2 & 1 & 1 & -1 & -1 \\
1 & 1 & 2 & 0 & -1 & -1 \\
1 & 1 & 0 & 2 &  0 & -1 \\
-1 & -1 & -1 &  0 & 2 & 1 \\
-1 & -1 & -1 & -1 & 1 & 2
\end{pmatrix}
\end{gather*}
are stabily isomorphic.  Since the discriminant group of both lattices is cyclic of prime order, we can easily see this by computing the determinant of both matrices.  Since $\cL'$ has rank~$6$ it does not split a hyperbolic plane.  When Algorithm~\ref{alg:positive_discriminant_forms} is applied to~$\cL$, one gets a lattice~$\cL''$ of rank~$14$, which is necessarily of the form $\cL'' \cong \cL' \oplus E_8$.
\end{example}


%% file: modularforms.tex
\section{Vector valued modular forms and Jacobi forms}
\label{sec:modularforms}

Given $N \in \ZZ_{>0}$, let
\begin{gather*}
  \HS_{1, N}
:=
  \HS \times \CC^N
\end{gather*}
be the Jacobi upper half plane, where $\HS = \{\tau = u + i v \,:\, v > 0\} \subset \CC$ is the Poincar\'e upper half plane.  It is acted on by the full Jacobi group
\begin{gather*}
  \Gamma^\rmJ
:=
  \SL{2}(\ZZ) \ltimes (\ZZ^2 \otimes \ZZ^N)
\text{,}
\end{gather*}
where the semidirect product is defined via the natural action of $\SL{2}(\ZZ)$ on $\ZZ^2$.  We write $\gamma^\rmJ = \big( \gamma, (\lambda, \mu) \big)$ for a typical element in $\Gamma^\rmJ$.  Here, $\lambda, \mu \in \ZZ^N$ are the columns of the second component of $\gamma^\rmJ$, which is viewed as a matrix via the natural identification $\ZZ^2 \otimes \ZZ^N \rightarrow \Mat{N, 2}{\ZZ}$.  As before, we write $\gamma = \left(\begin{smallmatrix}a & b \\ c & d\end{smallmatrix}\right)$ for a typical element in $\SL{2}(\ZZ)$.

The action $\Gamma^\rmJ \times \HS_{1, N} \rightarrow \HS_{1, N}$ is given by
\begin{gather}
  \gamma^\rmJ (\tau, z)
=
  \Big( \frac{a \tau + b}{c \tau + d}, \frac{z + \lambda \tau + \mu}{c \tau + d} \Big)
\text{.}
\end{gather}
This action extends to a family of actions on function $\phi :\, \HS_{1, N} \rightarrow \CC$.  They depend on both a weight $k \in \ZZ$ and a (Jacobi) index $L \in \Mat{N}{\ZZ}$ which is symmetric and has even diagonal entries.  In order to define them, denote $z^\tr L z$ by $L[z]$ ($z \in \CC^N$).
\begin{multline*}
  \big( \phi \big|_{k, L}\, \gamma^\rmJ \big) (\tau, z)
:=
\\[4pt]
  (c \tau + d)^{-k} e\big( -c L[z + \lambda \tau + \mu] \slashdiv (c \tau + d) + \tau L[\lambda] + 2 \lambda^\tr L z \big)
  \;
  \phi\big(\gamma^\rmJ (\tau, z) \big)
\text{.}
\end{multline*}
Here and throughout the paper, we use the notation $e(x) = \exp(2 \pi i\, x)$.

The following definition, which is folklore, includes a special case of a definition made in~\cite{Zi89}, where (nonweak) Jacobi forms of higher degree were defined.  It also generalizes the definition of weak Jacobi forms made in~\cite{EZ85} (the index~$m$ that was used there corresponds to~$L \slashdiv 2$ in our notation if $L$ has rank~$1$).
\begin{definition}
\label{def:jacobiforms}
A weakly holomorphic Jacobi form of weight $k \in \ZZ$ and index $L$ is a holomorphic function $\phi :\, \HS_{1, N} \rightarrow \CC$ that
\begin{enumerate}[(i)]
\item satisfies $\phi \big|_{k, L}\, \gamma^\rmJ = \phi$ for all $\gamma^\rmJ \in \Gamma^\rmJ$, and
\item has growth $\phi(\tau, \alpha \tau + \beta) = O\big( e^{a(\alpha, \beta) z} \big)$ for all $\alpha, \beta \in \QQ^N$ and some $a(\alpha, \beta) > 0$.
\end{enumerate}

A weakly holomorphic Jacobi form that satisfies $\phi(\tau, \alpha \tau + \beta) = O(1)$ for all $\alpha, \beta \in \QQ^N$ is called a Jacobi form.
\end{definition}
\noindent We denote the space of such weakly holomorphic Jacobi forms of weight $k$ and index $L$ by $\rmJ^!_{k, L}$.  The space of Jacobi forms is denoted by $\rmJ_{k, L}$.

We now define vector valued elliptic modular forms.  The metaplectic cover $\Mp{2}(\ZZ)$ of $\SL{2}(\ZZ)$ is the preimage of $\SL{2}(\ZZ)$ in $\Mp{2}(\RR)$, the connected double cover of $\SL{2}(\RR)$.  Write $\gamma = \left(\begin{smallmatrix} a & b \\ c & d \end{smallmatrix}\right)$ for a typical element of $\SL{2}(\RR)$.  The elements of $\Mp{2}(\RR)$ can be written as $\big(\gamma,\, \tau \mapsto \sqrt{c \tau + d}\big)$, where the first component is an element of $\SL{2}(\RR)$ and the second is a holomorphic function on $\HS$.  Since there are two branches of the square root, this yields indeed a double cover of $\SL{2}(\RR)$.

Given a representation $(\rho, V_\rho)$ of $\Mp{2}(\ZZ)$, $k \in \tfrac{1}{2}\ZZ$ and $f :\, \HS \rightarrow V_\rho$, we define
\begin{gather*}
  \big( f\big|_{k,\rho} \, (\gamma, \omega) \big)\, (\tau)
:=
  \omega(\tau)^{-2 k} \, \rho\big((\gamma, \omega)\big)^{-1} \,
  f\Big( \frac{a \tau + b}{c \tau + d} \Big)
\end{gather*}
for all $(\gamma, \omega) \in \Mp{2}(\ZZ)$.
\begin{definition}
Let $(\rho, V_\rho)$ be a finite-dimensional representation.  A weakly holomorphic vector valued modular form of type $\rho$ and weight $k \in \tfrac{1}{2}\ZZ$ is a holomorphic function $f \,:\, \HS \rightarrow V_\rho$ such that the following conditions are satisfied:
\begin{enumerate}
\item For all $\gamma \in \Mp{2}(\ZZ)$ we have $f|_{k,\rho} \, \gamma = f$.
\item We have $\| f(\tau) \| = O(e^{ay})$ for some $a > 0$ as $y \rightarrow \infty$, where $\| \,\cdot\, \|$ is some norm on $V_\rho$.
\end{enumerate}
\end{definition}

We will focus on Weil representations attached to discriminant forms~$D$.  In this case, the second condition translates into a condition on the Fourier expansion.  It is of the form
\begin{gather*}
  \sum_{-\infty \ll n \in \QQ} a(n) \, e(n \tau),
  \quad
  a(n) \in V_\rho
\text{.}
\end{gather*}

In order to define the Weil representation, write
\begin{gather*}
  T
:=
  \big( \left(\begin{smallmatrix}1 & 1 \\ 0 & 1\end{smallmatrix}\right),\, 1 \big)
\quad\text{and}\quad
  S
:=
  \big( \left(\begin{smallmatrix}0 & -1 \\ 1 & 0\end{smallmatrix}\right),\, \sqrt{\tau} \big)
\end{gather*}
for the generators of $\Mp{2}(\ZZ)$.  The root in the definition of $S$ is the principal branch.  The Weil representation $\rho_D$ is defined on the group algebra $\CC[D]$ with canonical basis elements~$\frake_\gamma$ ($\gamma \in D$).  The images of $T$ and $S$ under $\rho_D$ are as follows~\cite{Sk08}:
\begin{align}
  \rho_D (T) \, \frake_\gamma
&:=
  e\big( \gamma^2\big) \, \frake_\gamma
\text{,}
\\[6pt]
  \rho_D (S) \, \frake_\gamma
&:=
  \frac{1}{\sigma(D) \, \sqrt{|D|}}
  \sum_{\delta \in D}
  e\big(- (\gamma, \delta)\big) \frake_{\delta}
\text{,}
\quad
  \sigma(D)
:=
  \frac{1}{\sqrt{|D|}}
  \sum_{\gamma \in D} e\big(-\gamma^2 \big)
\text{.}
\end{align}
We denote the dual Weil representation by ${\check \rho}_{D}$.  Given a lattice $L$, we write $\rho_L$ and ${\check \rho}_L$ for $\rho_{\disc\, L}$ and ${\check \rho}_{\disc\, L}$, respectively.

Our interest in Jacobi forms stems from the next Theorem.  It relates Jacobi forms and vector valued elliptic modular forms via theta series.  Given $l \in \ZZ^N \slashdiv L \ZZ^N$, define
\begin{gather}
  \theta_{L, l}(\tau, z)
:=
  \sum_{\substack{r \in \ZZ^N \\ r \equiv l \pmod{L \ZZ^N}}}
   q^{L^{-1}[r] \slashdiv 2} \zeta^r
\text{.}
\end{gather}
\begin{theorem}[{\cite[p.~210]{Zi89}}]
\label{thm:thetadecomposition}
Let $\phi \in \rmJ^!_{k, L}$.  Then there is a vector valued weakly holomorphic modular form $(h_l)_{l \in \disc\, L}$ of type ${\check \rho}_L$ and weight $k - \frac{N}{2}$, such that
\begin{gather}
\label{eq:thetadecomposition}
  \phi(\tau, z)
=
  \sum_{l \in \disc\, L} \!\theta_{L, l}(\tau, z)\, h_l(\tau)
\text{.}
\end{gather}
Conversely, given a vector valued weakly holomorphic modular form of type ${\check \rho}_L$ and weight $k - \frac{N}{2}$, (\ref{eq:thetadecomposition}) defines an element of~$\rmJ^!_{k, L}$.
\end{theorem}
We denote the induced bijection between $\rmJ_{k + \frac{N}{2}, L}$ and $\rmM_{k}({\check \rho}_L)$ by $\Theta_L$, suppressing the weight.

We start revisiting several basic facts about Jacobi forms, which we will need to formulate Algorithm~\ref{alg:jacobiforms}.  First, recall the dimension formula for Jacobi forms, which follows immediately from~\cite{Bo00b}, where a dimension formula for vector valued modular forms was given.  In order to state the next theorem, note that vector valued modular forms of weight~$k - \frac{N}{2}$ and type~${\check \rho}_L$ are supported on the span of $\frake_{\mu} + \frake_{-\mu}$, $\mu \in \disc\, L$, if $k \equiv 0 \pmod{2}$, and on the span of $\frake_{\mu} - \frake_{-\mu}$, if $k \equiv 1 \pmod{2}$.  They vanish in all other cases, and hence we shall assume that $k \in \ZZ$.  Denote the restriction of ${\check \rho}_L$ to $\lspan \{\frake_\mu + \frake_{-\mu} \}$ or $\lspan \{\frake_\mu - \frake_{-\mu}\}$, respectively, by $\widetilde{\rho}_{k, L}$.

Given a unitary matrix $M$ with eigenvalues $e( \beta_i )$, $0 \le \beta_i < 1$, set $\alpha(M) = \sum_i \beta_i$.
\begin{theorem}[{\cite{Bo00b}, Lemma~7.3 and~\cite{Fi87}}]
\label{thm:dimension_for_jacobi_forms}
If $k \ge 2 + \frac{N}{2}$, the dimension of $\rmJ_{k, L}$ equals
\begin{gather*}
  \dim(\widetilde{\rho}_{k,L}) + \frac{\dim(\widetilde{\rho}_{k,L}) k}{12}
  - \alpha\big( e^{\frac{\pi i k}{2}}\, \widetilde{\rho}_{k, L}(S) \big)
  - \alpha\big( e^{\frac{-\pi i k}{3}}\, \widetilde{\rho}_{k, L}(S T)^{-1} \big)
  - \alpha\big( \widetilde{\rho}_{k, L}(T) \big)
\text{.}
\end{gather*}
\end{theorem}
\begin{remark}
For reasons of performance, it is not feasible to implement the above dimension formula without major modifications using exact arithmetic.  The implementation of Algorithm~\ref{alg:jacobiforms} discussed in Section~\ref{sec:implementation} makes use of the fact that the eigenvalues of $\widetilde{\rho}_{k, L}(S)$ and $\widetilde{\rho}_{k, L}(ST)$ are twelfth roots of unity.  Upper bounds for each multiplicity of these eigenvalues are computed using interval arithmetic and Householder transformations.  Correctness of the computations can then be proved a posteriori by comparing the sum of all multiplicities and the dimension of $\widetilde{\rho}_{k, L}$.
\end{remark}

To state the next lemma, which is folklore, set $L[s] = s^\tr L s$ for any pair of matrices $L$ and $s$ of compatible size.
\begin{lemma}
\label{la:restriction_of_jacobi_forms}
Let $\phi$ be a (weakly holomorphic) Jacobi form of weight~$k$ and index~$L$.  For $s \in \Mat{N, N'}{\ZZ}$ ($N' \in \ZZ_{> 0}$), the pullback $\phi[s](\tau, z) = \phi(\tau, s z)$ ($z \in \CC^{N'}$) is a (weakly holomorphic) Jacobi form of weight~$k$ and index~$L[s]$.
\end{lemma}
\begin{proof}
The conditions in Definition~\ref{def:jacobiforms} can be verified by a straight forward computation.
\end{proof}

A weakly holomorphic Jacobi form $\phi$ has a Fourier expansion of the following form:
\begin{gather}
\label{eq:jacobi_fourierexpansion}
  \phi(\tau, z)
=
  \sum_{n \in \ZZ, r \in \ZZ^N} c(n, r)\, q^n \zeta^r
\text{,}
\end{gather}
where $q = \exp(2 \pi i \tau)$ and $\zeta^r = \exp(2 \pi i \sum_{i = 1}^N r_i z_i)$.  Notice that the Fourier coefficients $c(n, r)$ of a Jacobi form vanish if $2 |L| n - L^{\#}[r] < 0$, where $|L| := \det(L)$ and $L^{\#} := |L| \cdot L^{-1}$ is the adjoint of $L$.

Because Jacobi forms are invariant under the action of $\big(I_2, (\lambda, 0)\big) \in \Gamma^\rmJ$, $\lambda \in \ZZ^N$, $I_2$~the identity matrix of size~$2$, the Fourier coefficients $c(n, r)$ satisfy the following relations, which are connected to the statement of Theorem~\ref{thm:thetadecomposition}.
\begin{align}
\label{eq:higherjacobiforms_fourier_relation1}
  c(n, r)
&=
  c(n + L[\lambda] \slashdiv 2 + r^\tr \lambda, r + L \lambda)
\\[4pt]
\label{eq:higherjacobiforms_fourier_relation2}
  c(n, -r)
&=
  (-1)^k \, c(n, r)
\text{.}
\end{align}

It is convenient to define
\begin{gather*}
  {\td c}( \Delta, r )
:=
  c\big( \frac{\Delta + L^\#[r]}{2 |L|},\, r \big)
\text{.}
\end{gather*}
By the above relations, ${\td c}$ only depends on $r \pmod{L \ZZ^N}$.


%% file: algorithm.tex
\section{Computing Jacobi forms}
\label{sec:algorithm}

Because of~(\ref{eq:higherjacobiforms_fourier_relation1}) and~(\ref{eq:higherjacobiforms_fourier_relation2}), the Fourier expansion~(\ref{eq:jacobi_fourierexpansion}) of a Jacobi form~$\phi$ is determined by all $c(n, r)$ where $r$ lies in a set $\cR$ of representatives of $\big( \ZZ^N \slashdiv L \ZZ^N \big) \slashdiv \{ \pm 1 \}$.  We can assume that the elements $r \in \cR$ have minimal norm among all elements of $r + L \ZZ^N$.  Set
\begin{gather*}
  {\rm neigh}(r)
=
  \{ r' \in r + L \ZZ^N \,:\, L[r'] = L[r] \}
\quad\text{and}\quad
  {\rm mult}(r)
=
  \# {\rm neigh}(r)
\text{.}
\end{gather*}

We call a finite set $\cS \subset \ZZ \times L \ZZ^{N}$ a complete set of restriction vectors (for $\cR$) if the matrix
\begin{gather*}
  F(\cR, \cS)
:=
  \big( {\rm mult}(r)\, \delta_{s^\tr r = {\td r}} \big)_{r \in \cR, ({\td r}, s) \in \cS}
\end{gather*}
has rank $\# \cR$.  We say that $\cS$ is minimal if the above matrix is invertible.  Obviously, every complete set of restriction vectors contains a minimal one.  By abuse of notation we will write $\cS$ for the projection of $\cS$ to its second component.
\begin{lemma}
\label{la:set_of_restriction_vectors}
Given~$L$ and a set $\cR \subset \ZZ^N$, there is a complete set of restriction vectors for $\cR$.
\end{lemma}
\begin{proof}
It suffices to prove that for each $r \in \cR$ there is $({\td r}, s) \in \ZZ \times L \ZZ^N$ such that $s^\tr r = {\td r}$ and $r$ is the unique element in $\cR$ with this property.  Choose $s' \in \RR^N$ whose components are linearly independent over $\QQ$.  Set 
\begin{gather*}
  \epsilon
=
  \min_{r' \ne r'' \in \cR} (|t_{r'} - t_{r''}|)
\text{,}
\end{gather*}
where $t_{r'} = {s'}^\tr r'$, $r' \in \cR$.  Choose an approximation $s'' \in \QQ^N$ of $s'$ such that
\begin{gather*}
  \big| (s' - s'')_i \big|
<
  \epsilon \slashdiv \Big( N \max_{\substack{r' \in \cR \\ 1 \le j \le N}} \big(|r'_j|\big) \Big)
\end{gather*}
for all $1 \le i \le N$.  Then all $s'' r' \in \QQ$, $r' \in \cR$ are distinct, and we can set $s = s'' \cdot {\rm lcm}\{\text{denominator of $s''_i$}\}_{1 \le i \le N}$.
\end{proof}
For later use, we need strong sets of restriction vectors $\cS^{\rm str}(r)$.
\begin{lemma}
\label{la:strong_set_of_restriction_vectors}
Let $r \in \cR \cup -\cR$.  Then there is a finite set $\cS^{\rm str}(r)$ that spans $\ZZ^N$ and has the following property.  Whenever $s^\tr r = s^\tr r'$ for some $s \in \cS(r)$ and some $r' \in \bigcup_{r \in \cR} {\rm neigh}(r) \cup {\rm neigh}(-r)$, then $r = r'$.  
\end{lemma}
\begin{proof}
The proof of lemma~\ref{la:set_of_restriction_vectors} gives us a vector $s$ which has the second property.  Fix a basis $v_1, \ldots, v_N$ of $\ZZ^N$.  Fix some $h > \max_{i, r'} |v_i^\tr (r - r')|$, where $r'$ runs through $\bigcup_{r \in \cR} {\rm neigh}(r) \cup {\rm neigh}(-r)$.  Then we can choose
\begin{gather*}
  \cS^{\rm str}(r)
:=
  \{ s,\, h s + v_1, \ldots, h s + v_N \}
\text{.}
\end{gather*}
This set obviously spans $\ZZ^N$, since $\{ v_1, \ldots, v_N \}$ does.  Since
\begin{gather*}
  |(h s + v_i)^\tr (r' - r)|
\ge
  h |s^\tr (r' - r)| - |v_i^\tr (r' - r)|
>
  0
\text{,}
\end{gather*}
its elements satisfy the second property given in the lemma.
\end{proof}

We are now in position to formulate an algorithm that computes the Fourier coefficients $c(n, r)$, $n < B \in \ZZ_{>0}$ associated to Jacobi forms of index~$L \in \Mat{N}{\ZZ}$,~$N \ge 2$.  To give a rigorous description, set
\begin{gather*}
  \cI(L, B)
=
  \big\{ (n, r) \in \ZZ \times \ZZ^N \,:\, 0 \le n < B,\, 2 |L| n - L^{\#}[r] \ge 0 \big\}
\text{,}
\end{gather*}
and let $\cF\cE_{L, B} :\, \bigoplus_{k} \rmJ_{k, L} \rightarrow \CC^{\cI(L, B)}$ be the map that sends a Jacobi form to its truncated Fourier expansion.  We denote the subspace of~$\CC^{\cI(L, B)}$ whose element satisfy the relations~\ref{eq:higherjacobiforms_fourier_relation1} and~\ref{eq:higherjacobiforms_fourier_relation2} by~$\CC^{\cI(L, B);\,\Gamma^\rmJ, k}$.  Given $v \in \CC^{\cI(L, B)}$, let $v[s] \in \CC^{\cI(L[s], B)}$ be the vector, whose $(n, {\td r})$-th component is given by $\sum_{r \in \ZZ^N,\, s^\tr r = {\td r}} v_{(n, r)}$.  This formal pullback map is made so that $\cF\cE_{L, B}(\phi)[s] = \cF\cE_{L[s], B}(\phi[s])$ for any $\phi \in \rmJ_{k, L}$,~$k \in \ZZ$.
\begin{algorithm}
\label{alg:jacobiforms}
Fix $k \in \ZZ$, $L$ as above, and $0 < B \in \ZZ$.  The following algorithm computes $\cF\cE_{L, B}\big( \rmJ_{k, L} \big)$.
\begin{enumerate}[1:]
\item Compute a set $\cR$ of representatives r of $\big( \ZZ^N \slashdiv L \ZZ^N \big) \slashdiv \{ \pm 1 \}$ with minimal norm~$L[r]$ in $r + L \ZZ^N$.
\vspace{0.5ex}
\item Compute a minimal complete set $\cS$ of restriction vectors for $\cR$
\vspace{0.5ex}
\item \label{it:alg:jacobiforms:adjust_precision}
  Set $B \leftarrow \max\big\{B,\, \max_s \{\dim(M_{k + L[s]},\, \dim \rmM_{k - 1 + L[s]}\} \big\}$.
\vspace{0.5ex}
\item \label{it:alg:jacobiforms:compute_scalar_jacobi_forms}
 For each $s \in \cS$, compute $\cJ_{k, L[s], B} := \cF\cE_{L[s], B} \big( \rmJ_{k, L[s]} \big) \cap \QQ^{\cI(L[s], B)}$.
\vspace{0.5ex}
\item Compute $\cJ_{k, L, B} := \big\{ (v[s])_{s \in \cS} \,:\, v \in \QQ^{\cI(L, B);\,\Gamma^\rmJ, k} \big\} \subseteq \bigoplus_{s \in \cS} \QQ^{\cI(L[s], B)}$.
\vspace{0.5ex}
\item \label{it:alg:jacobiforms:check}
      If the dimension of $\cJ_{k, L, B} \,\cap\, \bigoplus_{s \in \cS} \cJ_{k, L[s], B}$ does not equal $\dim \rmJ_{k, L}$, enlarge $\cS$ by the shortest possible vector with respect to $q_\cL$.  Continue with Step~\ref{it:alg:jacobiforms:adjust_precision}.
\vspace{0.5ex}
\item Compute the inverse image of $\cJ_{k, L, B} \cap \bigoplus_{s \in \cS} \cJ_{k, L[s], B}$ under $v \mapsto (v[s])_{s \in \cS}$.  It equals $\cF\cE_B\big(\rmJ_{k, L}\big)$.  Truncate it to the required precision.
\end{enumerate}
\end{algorithm}
\begin{remarks}
\begin{enumerate}[1)]
\item An implementation of this algorithm in Sage~\cite{sageticket-jacobi} is available.  For reasons of speed, this implementation represents $\cJ_{k, L, B} \cap \bigoplus_{s \in \cS} \cJ_{k, L[s], B}$ as the intersection of three spaces.  The reader is referred to the implementation for details.
\item A precise set $\cS$ which is necessary for Step~\ref{it:alg:jacobiforms:check} to pass, can be deduced from the proof of termination.  We do not make this explicit, since in practice any sensible choice of $\cS$ is sufficient.
\end{enumerate}
\end{remarks}
\begin{proof}[Proof of Theorem~\ref{thm:algorithm_existence}]
As explained in the introduction, by the results in~\cite{Bo98} and~\cite{Br12}, it is sufficient to compute Fourier expansions of vector valued modular forms of weight~$\frac{2 + n}{2}$ and type~${\check \rho}_{\cL}$.  Since only finitely many divisors $Z(m, \mu)$ are considered, these Fourier expansion need to be computed only up to a finite precision.  Such expansions correspond to Fourier expansions of Jacobi forms via Corollary~\ref{cor:positive_discriminant_forms} and Theorem~\ref{thm:thetadecomposition}.  Now Algorithm~\ref{alg:jacobiforms} can be used to compute the desired linear equivalences.
\end{proof}
\begin{proof}[Proof of correctness of Algorithm~\ref{alg:jacobiforms}]
We first prove correctness of Algorithm~\ref{alg:jacobiforms}.  The Fourier expansion of Jacobi forms in $\rmJ_{k, L}$ is determined by $c(n, r)$, where $r \in \cR$.  Since $\cS$ is contains a complete set of restriction vectors for $\cR$, the restriction map
\begin{gather*}
  \bigoplus_{s \in \cS} (\,\cdot\,) [s] :\,
  \rmJ_{k, L} \rightarrow \bigoplus_{s \in \cS} \rmJ_{k, L[s]}
\end{gather*}
is injective.  The condition $B \ge \max_s \{\dim \rmM_{k + L[s]},\, \dim \rmM_{k - 1 + L[s]}\}$ guarantees that the $\cF\cE_{L[s], B}$, $s \in \cS$ are injective.

We exploit the commutative diagram
\begin{gather*}
\xymatrix{
  \rmJ_{k, L}  \ar[rr]^{\cF\cE_{L, B}} \ar[d]_{\bigoplus (\,\cdot\,)[s]} && \cJ_{k,L,B} \ar[d]_{\bigoplus (\,\cdot\,)[s]} \\
  \bigoplus_{s \in \cS} \rmJ_{k, L[s]} \ar[rr]^{\bigoplus \cF\cE_{L[s], B}} && \bigoplus_{s \in \cS} \cJ_{k, L[s]}
}
\end{gather*}
By what we have shown the composition of the left arrow with the bottom arrow is injective.  Hence so is the top arrow.  By the dimension check in Step~\ref{it:alg:jacobiforms:check}, it is an isomorphism.  This proves correctness.
\end{proof}

The proof that Algorithm~\ref{alg:jacobiforms} terminates will occupy the rest of this section.  We will denote the components of $z$ by $z_1,\ldots, z_N$.  We say that a function $\CC^N \rightarrow \CC$ is elliptic with respect to $L$ and $\tau \in \HS$ if $\phi$ is invariant under $\big|_{L, \tau} (\lambda, \mu)$ for all $\lambda, \mu \in \ZZ^N$.  Here,
\begin{gather*}
  \phi \big|_{L, \tau} (\lambda, \mu)
:=
  \big( (\tau', z) \mapsto \phi(z) \big)
  \big|_{k, L} \big( I_2, (\lambda, \mu) \big) \Big|_{\tau' = \tau}
\end{gather*}
for some $k \in \ZZ$---the slash action on the right hand side is independent of weight~$k$.  We will henceforth use the language of lattices, that was introduced in Section~\ref{sec:discriminantforms}.  Then $L$ will be a fixed Gram matrix of a lattice $\cL$.
\begin{lemma}
\label{la:vanishing_z1_of_higher_jacobi_forms}
Let $\phi$ be a holomorphic function that is elliptic with respect to a lattice~$\cL$ and $\tau \in \HS$.  Then there is a positive integer $b$ that only depends on $L$ such that if
\begin{gather}
\label{eq:vanishing_z1_of_higher_jacobi_forms}
  \CC \rightarrow C^\infty(\HS_{1, N-1})
\text{,}\,
  z_1 \mapsto \phi\big(\tau, (z_1, \ldots, z_N)\big)
\end{gather}
has more than $b$ zeros counted with multiplicity, then $\phi = 0$.
\end{lemma}
\begin{proof}
Choose a vector $l = (l_1, 0, \ldots, 0)^\tr \in L^{-1} \ZZ^N \cap \ZZ^N$.  Since $\phi$ is elliptic with respect to $\cL$ and $\tau$, it is, in particular, invariant under the transformations $\big|_{L, \tau} \big(I_2, (0, e_1)\big)$ and $\big|_{L, \tau} \big(I_2, (l, 0)\big)$, where $e_1$ is the first canonical basis vector of $\ZZ^N$.  In other words, the function~(\ref{eq:vanishing_z1_of_higher_jacobi_forms}) is elliptic in the usual sense.  Then the result follows by applying the Residue Theorem in the same way as in~\cite{EZ85}, Theorem~1.2.
\end{proof}

\begin{lemma}
\label{la:vanishing_s_of_higher_jacobi_forms}
Let $\phi$ be a holomorphic function that is elliptic with respect to a lattice $\cL$ and $\tau \in \HS$.  Suppose that $N \ge 2$.  Then there is a set finite $\cS$ of matrices $s \in \Mat{N, N-1}{\ZZ}$ that only depends on $\cL$ such that if $\phi[s] = 0$ for all $s \in \cS$, then $\phi$ vanishes identically.
\end{lemma}
The following corollary can be deduced from Lemma~\ref{la:vanishing_s_of_higher_jacobi_forms} using induction.
\begin{corollary}
\label{cor:vanishing_s_of_higher_jacobi_forms}
Let $\phi$ be an elliptic function for a finite index sublattice $L'$ of~$L$, and suppose that $N \ge 2$.  There is a set finite $\cS$ of vectors in $\ZZ^{N - 1}$ that only depends on $L'$ such that if $\phi[s] = 0$ for all $s \in \cS$, then $\phi = 0$.
\end{corollary}
\begin{proof}[Proof of Lemma~\ref{la:vanishing_s_of_higher_jacobi_forms}]
We will show that, given $\widetilde{z}_1 \in \QQ + i \QQ$, there is a set $\cS'(\widetilde{z}_1)$ of matrices $s \in \Mat{N, N-1}{\ZZ}$ such that if $\phi[s] = 0$ for all $s \in \cS'(\widetilde{z}_1)$, then
\begin{gather*}
  \phi_{\widetilde{z}_1}\big( z_2, \ldots, z_N \big)
:=
  \phi\big( \widetilde{z}_1, z_2, \ldots, z_N \big)
=
  0
\text{.}
\end{gather*}
Suppose that we have already proved that $\cS'(\widetilde{z}_1)$ exits for all $\widetilde{z}_1$.  Apply Lemma~\ref{la:vanishing_z1_of_higher_jacobi_forms} to $\phi$, $\cL$ and $\tau$ to obtain $b$.  Then we can set $\cS := \bigcup \cS'(\widetilde{z}_1)$ where the union runs over $b$ distinct $\widetilde{z}_1$'s.

We are hence reduced to finding $\cS'(\widetilde{z}_1)$.  In the case of $\widetilde{z}_1 = 0$, we can choose
\begin{gather*}
  \cS'(\widetilde{z}_1)
=
 \Big \{
  \begin{pmatrix}
   \big( 0^{(N - 1)} \big)^\tr \\
   I_{N - 1}
  \end{pmatrix}
  \Big \}
\text{,}
\end{gather*}
where $0^{(N - 1)}$ is the vector of length $N - 1$ whose entries are all $0$.  To prove the other cases, note that $\phi_{\widetilde{z}_1}$ is elliptic for some finite index sublattice of $\cL$.  We use induction on $N \ge 2$.  We consider the case $N = 2$ separately.  We employ Lemma~\ref{la:vanishing_z1_of_higher_jacobi_forms} to get $b(\widetilde{z}_1)$ that can be applied to $\phi_{\widetilde{z}_1}$.  We can set $\cS'(\widetilde{z}_1) = \big\{ s'_i \big\}_{1 \le i \le b(\widetilde{z}_1)} \subset \QQ$, where the $s'_i$ are pairwise not equivalent modulo~$1$.  To see that that this choice of~$\cS'(\widetilde{z}_1)$ works, write $s'_i = p_i \slashdiv q_i$ with coprime integers $p_i$ and $q_i$.  We have
\begin{gather*}
  \phi[(q_i, p_i)^\tr]( \widetilde{z}_1 \slashdiv q_i )
=
  \phi( \widetilde{z}_1, \widetilde{z}_1 s'_i )
=
  \phi_{\widetilde{z}_1} ( \widetilde{z}_1 s'_i )
\text{.}
\end{gather*}
Suppose that the left hand side vanishes for all $i$.  Then so does the right hand side.  By Lemma~\ref{la:vanishing_z1_of_higher_jacobi_forms}, we then find that $\phi_{\widetilde{z}_1} = 0$.  This completes the proof in the case of~$N = 2$.

In the case of $N > 2$ we can assume that the statement of the lemma holds for all $N' < N$.  Given $\widetilde{z}_1$, choose $\cS$ for $\phi_{\widetilde{z}_1}$ as in the lemma and call it $\cS(\widetilde{z}_1)$.  A straightforward computation shows that we can set
\begin{gather*}
  \cS'(\widetilde{z}_1)
=
  \Big\{
    \begin{pmatrix}
    1         & ( 0^{(N - 2)} \big)^\tr \\
    0^{(N - 1)} & s'                   
    \end{pmatrix}
  \Big\}_{s \in \cS(\widetilde{z}_1)}
\text{.}
\end{gather*}
\end{proof}

\begin{proposition}
\label{prop:modularity_s_of_higher_jacobi_forms}
Let $\phi(\tau, z)$ be a function that is $1$-periodic in $\tau \in \HS$ and an elliptic function in $z$ with respect to $\cL$ and $\tau$ for all $\tau$.  Fix a weight~$k \in \ZZ$, and suppose that $N \ge 2$.  Then there is a set $\cS \subset \ZZ^{N - 1}$ that only depends on $\cL$ such that if $\phi[s]$ is a Jacobi form of weight~$k$ and index $L[s]$ for all $s \in \cS$, then $\phi$ is a Jacobi form of weight $k$ and index $L$.
\end{proposition}
\begin{proof}
Set $S = \left(\begin{smallmatrix}0 & -1 \\ 1 & 0\end{smallmatrix}\right)$.  Using the group structure of $\Gamma^\rmJ$, we see that $\big( \phi - \phi\big|_{k, L} S \big)$ is elliptic in $z$ with respect to $\cL$ and $\tau$ for all $\tau$.  Choose $\cS$ as in Corollary~\ref{cor:vanishing_s_of_higher_jacobi_forms} and assume that $\phi[s] - \phi[s] \big|_{k, L} S = 0$ for all $s \in \cS$.  For $z' \in \CC^{N - 1}$, we have
\begin{gather*}
  \big( \phi - \phi\big|_{k, L} S \big)[s] (\tau, z')
=
  \phi[s](\tau, z') - \tau^{-k} \phi\big(\frac{-1}{\tau}, \frac{s z}{\tau} \big)
=
  \big( \phi[s] - \phi[s] \big|_{k, L} S \big) (\tau, z')
=
  0
\text{.}
\end{gather*}
Hence, by Corollary~\ref{cor:vanishing_s_of_higher_jacobi_forms}, we have $\phi - \phi\big|_{k, L} S = 0$.  Since $\Gamma^\rmJ$ is generated by $S$, $T$, and $\ZZ^2 \otimes \ZZ^N \subset \Gamma^\rmJ$, this proves that $\phi$ is a Jacobi form of weight~$k$ and index~$L$.  We have thus proved that the proposition holds with $\cS$ as in Corollary~\ref{cor:vanishing_s_of_higher_jacobi_forms}.
\end{proof}

\begin{proof}[Proof of termination of Algorithm~\ref{alg:jacobiforms}]
Given $B' \le B$, let $\cU_{B', B} = \lspan\big( e_{n, r} \,:\, B' \le n < B \big) \subset \QQ^{\cI(L, B)}$, where $e_{n, r}$ denotes the canonical basis vector attached to $(n, r) \in \cI(L, B)$.  Set
\begin{gather*}
  \cJ_{k, L}
:=
  \lim_{0 < B} \bigcap_{0 < B' \le B} \cJ_{k, L, B'} + \cU_{B', B}
\subset
  \QQ^{\cI(L, \infty)}
\text{.}
\end{gather*}
Since $\dim \cJ_{k, L, B'} < \infty$ for all $B$ and $\dim \cJ_{k, L, B'} \le \dim \cJ_{k, L, B}$ for $B' \le B$ sufficiently large, it suffices to show that
\begin{gather*}
  \cF\cE_\infty :\, J_{k, L} \hookrightarrow \cJ_{k, L}
\end{gather*}
is an isomorphism.

The Fourier coefficients of $J_{k, L[s]}$ have at most polynomial growth in $n$, and so do the coefficients of any element $\phi$ of $\cJ_{k, L}$.  Consequently, any such $\phi$ corresponds to an absolutely convergent Fourier expansion on $\HS_{1, N}$.  By construction these functions are elliptic in $z$ and $1$-periodic on $\tau$.  Lemma~\ref{prop:modularity_s_of_higher_jacobi_forms} allows us to find a set $\cS_{k,L}$, so that it is sufficient to check that $\phi[s]$ is modular.  Since in Step~\ref{it:alg:jacobiforms:check} $\cS$ is enlarged by short vectors, it will eventually contain $\cS_{k, L}$.  This proves that Algorithm~\ref{alg:jacobiforms} terminates.
\end{proof}


%% file: heckeoperators.tex
\section{Hecke operators and newforms}
\label{sec:heckeoperators}

We define three types of Hecke operators, previously given in~\cite{Gr94}, that generalize those studied in~\cite{EZ85} in the case $N = 1$.  The aim of this section is to relate two of them to an action on vector valued modular forms.  We give the third one only for completeness.  The first operator is defined for all $s \in \Mat{N}{\ZZ} \cap \GL{N}(\QQ)$.  The second and third are defined for all $0 < l \in \ZZ$ that are coprime to $|L|$.

Expressions for them can be obtained from the parabolic Hecke algebra.  In order to define it, consider the group
\begin{gather*}
  \GSp{N + 1}(\ZZ)
=
  \big\{ g \in \Mat{2(N + 1)}{\ZZ} 
         \,:\, M^\tr J M = \nu(g) J,\; \nu(g) \in \ZZ \big\}
\text{,}
\\[4pt]
  J
=
  \begin{pmatrix}
    0^{(N + 1, N + 1)} & - I_{N + 1} \\
    I_{N + 1}         & 0^{(N + 1, N + 1)}
  \end{pmatrix}
\text{,}
\end{gather*}
where $0^{(N_1, N_2)} \in \Mat{N_1, N_2}{\ZZ}$ is the zero matrix.  The Jacobi group~$\Gamma^\rmJ_N$ can be embedded into $\Sp{N + 1}(\ZZ) = \{ g \in \GSp{N + 1}(\ZZ) \,:\, \nu(g) = 1 \}$:
\begin{gather*}
  \Big( \left(\begin{smallmatrix} a & b \\ c & d \end{smallmatrix}\right),\,
        (\lambda, \mu) \Big)
\mapsto
  \begin{pmatrix}
    a        & 0^{(1,N)} & b & a \mu^\tr - b \lambda^\tr \\
    \lambda & I_N       & \mu & 0^{(N, N)} \\
    c        & 0^{(1, N)} & d    & c \mu^\tr - d \lambda^\tr \\
    0^{(N, 1)} & 0^{(N, N)} & 0^{(N, 1)} & I_N
  \end{pmatrix}
\end{gather*}
It sits inside the monoid~$\rmG\Gamma^\rmJ_N$ of all matrices
\begin{gather*}
  \begin{pmatrix}
    a        & 0^{(1,N)} & b & \mu^\tr \\
    \lambda & a'       & \mu & 0^{(N, N)} \\
    c        & 0^{(1, N)} & d    & -\lambda^\tr \\
    0^{(N, 1)} & 0^{(N, N)} & 0^{(N, 1)} & d'
  \end{pmatrix}
\in
  \GSp{N + 1}(\ZZ)
\text{,}
\end{gather*}
where $a', d' \in \Mat{N}{\ZZ} \cap \GL{N}(\QQ)$, $\mu, \lambda \in \ZZ^N$, and $\left(\begin{smallmatrix} a & b \\ c & d \end{smallmatrix}\right) \in \Mat{2}{\ZZ} \cap \GL{2}(\QQ)$.  The action of such a matrix, call it $g^\rmJ$ for the time being, on Jacobi forms is
\begin{multline*}
  \nu(\gamma^\rmJ)^{\frac{-(N + 1) k}{2}} \cdot \big( \phi \big|_{k, L}\, g^\rmJ \big) (\tau, z)
=
\\[4pt]
  (c \tau + d)^{-k}
  e\Big( -(c \tau + d)^{-1} L[z + \lambda \tau + \mu]
        + 2 \lambda^\tr L z + L[\lambda] \tau \Big)
  \phi\Big( \frac{a \tau + b}{c \tau + d},\, \frac{z + \lambda \tau + \mu}{c \tau + d} \Big)
\text{.}
\end{multline*}

The Jacobi Hecke algebra $\cH^\rmJ_N$ is the algebra attached to the Hecke pair $\big( \Gamma^\rmJ_N,\, \rmG\Gamma^\rmJ_N \big)$ -- see~\cite{Kr90, Gr94} for details.  Note that it is a noncommutative algebra.  The operator~$U'_s$ is induced by the double coset that is generated by the block diagonal matrix
\begin{gather*}
  \diag\big( \det(s), \det(s) s, \det(s), s^\# \big)
\text{.}
\end{gather*}
It maps Jacobi forms of index~$L$ to Jacobi forms of index $L[s]$.  Given $l$ as above, the operator~$V'_l$ is induced by the set of double cosets which come from the matrices
\begin{gather*}
  \begin{pmatrix}
    a        & 0^{(1,N)} & b & 0^{(1,N)}\\
    \lambda & l I_N       & 0^{(N,1)} & 0^{(N, N)} \\
    c        & 0^{(1, N)} & d    & 0^{(1,N)} \\
    0^{(N, 1)} & 0^{(N, N)} & 0^{(N, 1)} & I_N
  \end{pmatrix}
\text{,}
\end{gather*}
where $\left(\begin{smallmatrix} a & b \\ c & d \end{smallmatrix}\right)$ runs through all matrices with determinant $l$.  This operator maps Jacobi forms of index~$L$ to Jacobi forms of index $l L$.  The operator $T'_l$ preserves the index.  It is induced by the set of double cosets which come from matrices
\begin{gather*}
  \begin{pmatrix}
    a        & 0^{(1,N)} & b & 0^{(1,N)}\\
    \lambda & l I_N       & 0^{(N,1)} & 0^{(N, N)} \\
    c        & 0^{(1, N)} & d    & 0^{(1,N)} \\
    0^{(N, 1)} & 0^{(N, N)} & 0^{(N, 1)} & l I_N
  \end{pmatrix}
\text{,}
\end{gather*}
where $\left(\begin{smallmatrix} a & b \\ c & d \end{smallmatrix}\right)$ runs through all matrices with determinant $l^2$ such that $\gcd\, \gamma$ is a square.

We will use rescaled versions of $U'_s$, $V'_l$, and $T'_l$.  Their actions are defined as follows.
\begin{align}
\label{eq:def_hecke_us}
  \phi \big|_{k, L}\, U_s
& =
  \phi(\tau, s z)
\\[6pt]
\label{eq:def_hecke_vl}
  \phi \big|_{k, L}\, V_l
& =
  l^k \sum_{\substack{ \gamma \in \Gamma \backslash \Mat{2}{\ZZ} \\ \det \gamma = l}}
    (c \tau + d)^{-k} e\Big( \frac{- c L[z]}{c \tau + d} \Big) \,
    \phi\Big( \gamma \tau, \frac{l z}{c \tau + d} \Big)
\\[6pt]
\label{eq:def_hecke_tl}
  \phi \big|_{k, L}\, T_l
& =
  l^{k - 2(N + 1)}
  \sum_{\substack{\gamma \in \Gamma \backslash \Mat{2}{\ZZ}  \\ \det{\gamma} = l^2 \\ \gcd{\gamma} = \square}}\;
  \sum_{x \in \ZZ^2 \otimes \ZZ^N \slashdiv l (\ZZ^2 \otimes \ZZ^N)}
   \phi \big|_{k, L}\, (\gamma, x)
\text{.}
\end{align}
Hecke operators for classical elliptic modular forms will occur later.  By abuse of notation, we denote them also by $T_l$.

If $\phi$ is a Jacobi form of index $L$, then $\phi |_{k, L}\, U_s$ has index $L' := L[s]$.  In the language of lattices, the relation of indices corresponds to the inclusion $\cL' := s \cL \subseteq \cL$.  On the other hand, inclusions of lattices manifest themselves as isotropic subgroups~$H$ of $D' = \disc\, \cL'$.  We now show that the action of $U_s$ on the components of the theta decomposition is the same as the action of $\uparrow_H^{D'}$ defined in~\cite{Sch11}.  Recall the definition of this operator.  Let $H \subset D'$ be an isotropic subgroup of a discriminant form~$D'$.  Then, by definition, we have
\begin{gather*}
  f \big\uparrow_H^{D'}
=
  \sum_{\mu \in H^\perp \subseteq D'} \!\! f_{\mu + H} \, \frake_\mu
\end{gather*}
for every vector valued modular form~$f$ of type~$\rho_{D}$ with $D = H^\perp \slashdiv H$.  In the present situation, we can choose the isotropic subgroup $H := \cL \slashdiv \cL' \subset \disc\, \cL'$.  We then have $H^\perp \slashdiv H \cong \disc\, \cL \subseteq \disc\, \cL'$.
\begin{proposition}
\label{prop:hecke_us_commutation}
The following diagram commutes:
\begin{gather*}
\xymatrix{
  \rmJ_{k, L} \ar[r]^{U_s} \ar[d]^{\Theta_{L}} & \rmJ_{k, L'} \ar[d]^{\Theta_{L'}} \\ 
  \rmM_{k - \frac{N}{2}} \big({\check \rho}_{\cL}\big) \ar[r]^{\uparrow_{H}^{D'}} & \rmM_{k - \frac{N}{2}} \big( {\check \rho}_{\cL'} \big)
}
\end{gather*}
\end{proposition}
\begin{proof}
From~(\ref{eq:def_hecke_us}) and the theta decomposition of $\phi \in \rmJ_{k, L}$
\begin{gather*}
  \phi(\tau, z)
=
  \sum_{\mu \in \ZZ^N \slashdiv L \ZZ^N} \theta_{L, \mu}(\tau, z) \, h_\mu(\tau)
\end{gather*}
we see that it suffices to prove the equality
\begin{gather}
\label{eq:prop:hecke_us_commutation}
  \theta_{L, \mu} (\tau, s z)
=
  \sum_{\mu' \in \ZZ^N \slashdiv s \ZZ^N} \theta_{L', s^\tr \mu + s^\tr L \mu'} (\tau, z)
\text{.}
\end{gather}
By passing to the Fourier expansion of $\theta_{L, \mu}$, this becomes an easy calculation.
\begin{align*}
  \sum_{r \equiv \mu \pmod{L \ZZ^N}} \hspace{-1em}
  q^{L^{-1}[r] \slashdiv 2} \zeta^{s^\tr r}
&=
  \sum_{s^\tr r \equiv s^\tr \mu \pmod{s^\tr L \ZZ^N}} \hspace{-1.5em}
  q^{L[s]^{-1}[s^\tr r] \slashdiv 2} \zeta^{s^\tr r}
\\
&=
  \sum_{\substack{\mu' \in \ZZ^N \slashdiv s \ZZ^N \\ r \equiv s^\tr \mu + s^\tr L \mu' \pmod{L[s] \ZZ^N}}} \hspace{-1.5em}
  q^{L[s]^{-1}[r] \slashdiv 2} \zeta^{r}
\text{.}
\end{align*}
\end{proof}

We next treat the operator $V_l$.  In the case $N = 1$, the action of $V_l$ on the Taylor expansion of $\phi$ is given by usual Hecke operators.  This result can be extended to arbitrary $N$.  In order to make a precise statement, we recall the operator $\cD_\nu$ given in~\cite{EZ85} for $0 \le \nu \in \ZZ$.  On Jacobi forms~$\phi$ of even weight~$k$ and index~$m$ it equals
\begin{gather*}
  \cD_{2 \nu} (\phi) \text{,}
:=
  \sum_{0 \le \mu \le \nu}
  \frac{(-1)^\mu \, (2 \nu)! \, (k + 2 \nu - \mu - 2)!}
       {\mu! \, (2 \nu - 2 \mu)! \, (k + \nu - 2)!} \;
  \big( \frac{\partial_z}{2 \pi i} \big)^{2 \nu - 2 \mu}
  \big( \frac{\partial_\tau}{2 \pi i} \big)^\mu \phi
\text{.}
\end{gather*}
If $k$ is odd, set
\begin{gather*}
  \cD_{2 \nu + 1} (\phi) \text{,}
:=
  \sum_{0 \le \mu \le \nu}
  \frac{(-1)^\mu \, (2 \nu + 1)! \, (k + 2 \nu - \mu - 1)!}
       {\mu! \, (2 \nu + 1 - 2 \mu)! \, (k + \nu - 1)!} \;
  \big( \frac{\partial_z}{2 \pi i} \big)^{2 \nu + 1 - 2 \mu}
  \big( \frac{\partial_\tau}{2 \pi i} \big)^\mu \phi
\text{.}
\end{gather*}
In all cases that are not covered by the above definition we have $\cD_\nu(\phi) := 0$.

\begin{proposition}
\label{prop:hecke_vl_commutation}
For all $\Mat{N, N'}{\ZZ}$, $1 \le N' \le N$, we have
\begin{gather*}
  \phi[s] \big|_{k, L[s]}\, V_l
=
  \big( \phi \big|_{k, L[s]}\, V_l \big)[s]
\text{.}
\end{gather*}
In particular, the following diagram commutes for all $s \in \ZZ^N$ and all $0 \le \nu \in \ZZ$:
\begin{gather*}
\xymatrix{
  \rmJ_{k, L} \ar[d]_{\cD_\nu \,\circ\, (\,\cdot\,)[s]} \ar[rr]^{V_l} && \rmJ_{k, l L} \ar[d]^{\cD_\nu \, \circ\,  (\,\cdot\,)[s]} \\
  \rmM_{k + \nu} \ar[rr]^{T_l} && \rmM_{k + \nu}
}
\end{gather*}
\end{proposition}
\begin{proof}
The first claim follows immediatelly from the definition of $V_l$.  The second claim is then a consequence of the corollary on page~45 of~\cite{EZ85}.
\end{proof}

The action of $V_l$ on vector valued modular forms is involved.  Let $\cL'$ be a lattice with Gram matrix $l L$, that is, a rescaled version of $\cL$.  When identifying both $\cL^\#$ and $(\cL')^\#$ with $\ZZ^N$, then the multiplication by~$l$ map induces an embedding of the module $D = \disc\, \cL$ into $D' = \disc\, \cL'$;  The quadratic form is not preserved.  Indeed, when identifying $\mu \in \disc\, \cL$ with a corresponding vector in $\ZZ^N$, then $l \mu$ is unique up $l L \ZZ^N$.  Hence $l \mu$ corresponds to a well-defined element of $\disc\, \cL'$.  We have
\begin{gather*}
  q_\cL(\mu)
\equiv
  L^{-1}[\mu]
=
  (l L)^{-1} [l \mu] \slashdiv l
\equiv
  q_{\cL'}(l \mu) \slashdiv l
\text{.}
\end{gather*}
There are further multiplication maps from $D$ to $D'$, which are only well-defined up to submodules of $\ZZ^N \slashdiv l \ZZ^M$.  Let $d \isdiv l$, then $d \mu$ gives an element of $D'$ that is well-defined up to elements~$\mu'$ in the kernel $\ker( \, \cdot\, \frac{l}{d}\, |_{D'} )$ of the multiplication by $d$ map on $D'$.

Define an operator that restricts Fourier expansions on $\HS$:
\begin{gather}
  \sum_{n \in \QQ} a(n) q^n \big|\, {\rm res}(d, n')
=
  \sum_{n \in d \ZZ + n'} a(n) \, q^n
\text{.}
\end{gather}
With this restriction operator at hand, we can define a new Hecke operator for vector valued modular forms.  Let $D$ be a discriminant form and $D \cong \disc\, \cL$ for some positive definite lattice~$\cL$.  Then we set
\begin{gather}
\label{eq:def:hecke_vlL_vector_valued}
  f\big|\, \mathrm{sc}_l(\cL)
=
  \sum_{d \isdiv l}
  \big( \frac{l}{d} \big)^{k - 1}
  \sum_{\substack{ \mu \in \disc\, \cL \\ \mu' \in \ker( \, \cdot d\, |_{D'} )}}
    \Big( f_{\mu} |\, {\rm res}\big(d, \tfrac{d^2}{l} q_{D'}(\tfrac{l}{d} \mu + \mu') \big) \Big)  \big( l \tau \slashdiv d^2 \big) \;
  \frake_{\frac{l}{d} \mu + \mu'}
\text{.}
\end{gather}

At this stage, ${\rm sc}_l( \cL )$ does not only depend on $D$ but also on a lift of $D$ to a lattice~$\cL$.
\begin{proposition}
The following diagram computes:
\begin{gather*}
\xymatrix{
  \rmJ_{k, L} \ar[rr]^{V_l} \ar[d]^{\Theta_L} && \rmJ_{k, L'} \ar[d]^{\Theta_{L'}} \\
  \rmM_{k - \frac{N}{2}}({\check \rho}_{\cL}) \ar[rr]^{\mathrm{sc}_l(\cL)} && \rmM_{k - \frac{N}{2}}({\check \rho}_{\cL'})
}
\end{gather*}
\end{proposition}
\begin{proof}
We choose $\left(\begin{smallmatrix} a & b \\ 0 & d \end{smallmatrix}\right)$, $ad = l$, $b \pmod{d}$ as representatives for $\gamma \in \Gamma \backslash \Mat{N}{\ZZ}$, $\det(\gamma) = l$.  Then it suffices to prove for each $d \isdiv l$ that
\begin{align*}
& {}
  \sum_{\substack{r \equiv \mu \pmod{L \ZZ^N} \\ 0 \le \Delta \in \QQ}}
  \Big(
    \sum_{b \pmod{d}}
    e\big( \tfrac{b}{d} \big( L^{-1}[r] \slashdiv 2 + \Delta \big) \big)
  \Big)
    {\td c}(\Delta, \mu) \,
    q^{\frac{l}{d^2} ( L^{-1}[r] \slashdiv 2 + \Delta )} \zeta^{\frac{l}{d^2} r}
\\[4pt]
= & {}
  \sum_{\mu' \in \ZZ^N \slashdiv d \ZZ^N}
  \sum_{\substack{ r \equiv \frac{l}{d} ( \mu + L \mu') \pmod{l L \ZZ^N} \\ \Delta \in d \ZZ - L^{-1}[\mu + L \mu'] \slashdiv 2 }} \hspace{-1em}
    {\td c}(\Delta, \mu) \,
    q^{(l L)^{-1}[r] \slashdiv 2 + \frac{l}{d^2} \Delta} \zeta^r
\text{.}
\end{align*}
The Fourier coefficients ${\td c}$ of a Jacobi forms were introduced in Section~\ref{sec:modularforms}.  The necessary calculations are analogous to those in the proof of Proposition~\ref{prop:hecke_us_commutation}.
\end{proof}

We end this section with a decomposition into newspaces which is finer than the one employed in the next section.  Set
\begin{align*}
  \rmJ_{k, L}^{UV \rm old}
& =
  \sum_{\substack{{\td s} \in \Mat{N}{\ZZ} \\ L = L'[{\td s}]}} \,
  \sum_{\substack{l \isdiv L' \\ l \ne 1 \text{ or } \det({\td s}) \ne 1}} \!\!\!
  \rmJ_{k, L'} \big|_{k, L' \slashdiv l}\, V_l U_{\td s}
\text{,}
\\[6pt]
  \rmJ_{k, L}^{UV \rm new}
& =
  \big( \rmJ_{k,L}^{UV \rm old} \big)^\perp
\subseteq
  \rmJ_{k, L}
\text{,}
\end{align*}
where $l \isdiv L'$ means that $L' \slashdiv l \in \Mat{N}{\QQ}$ has integral entries and even diagonal entries, and ${\td s} = I_N$, if $\det({\td s}) = 1$.  The decomposition which generalized the newform decomposition in~\cite{EZ85} in the most straight forward way is
\begin{gather}
\label{eq:newform_decomposition}
  J_{k, L}
=
  \sum_{\substack{{\td s} \in \Mat{N}{\ZZ}\\ L = L'[{\td s}]}}\,
  \sum_{l \isdiv L'} \,
  \rmJ_{k, L' \slashdiv l}^{UV \rm new}
\text{,}
\end{gather}
where again ${\td s} = I_N$, if $\det({\td s}) = 1$.  We show that this decomposition is not a direct sum by giving a counter example.  Suppose it was.  Then one can recursively compute the dimension of $\rmJ_{k, L}^{UV \rm new}$.  Thus recursively define
\begin{gather*}
  \rmd^{UV \rm new}_{k, L}
=
  \dim \rmJ_{k, L}
  -
  \sum_{\substack{{\td s} \in \Mat{N}{\ZZ} \\ L = L'[{\td s}]}} \,
  \sum_{\substack{l \isdiv L' \\ l \ne 1 \text{ or } \det({\td s}) \ne 1}} \!\!\!
  \rmd^{UV \rm new}_{k, L' \slashdiv l}
\text{,}
\end{gather*}
where once more ${\td s} = I_N$, if $\det({\td s}) = 1$.
\begin{proposition}
Let $L = \left(\begin{smallmatrix} 8 & 4 \\ 4 & 8 \end{smallmatrix}\right)$ and $k = 4$.  Then we have $\rmd^{UV \rm new}_{k, L} = -1 < 0$.  In particular, for these values of $L$ and $k$, (\ref{eq:newform_decomposition}) is not a direct sum.
\end{proposition}
\begin{proof}
The corresponding code can be found at~\cite{raumhomepage}.
\end{proof}

%% file: newformsalgorithm.tex
\section{Modification of Algorithm~\ref{alg:jacobiforms}}
\label{sec:newforms_algorithm}

In this section, we give an improved version of Algorithm~\ref{alg:jacobiforms}, which makes use of a decomposition of $\rmJ_{k, L}$.  In total analogy to Section~3.1 of~\cite{Br12}, we set
\begin{align*}
  \rmJ_{k, L}^{U \rm old}
& =
  \sum_{\substack{{\td s} \in \Mat{N}{\ZZ},\; \det({\td s}) \ne 1 \\ L = L'[{\td s}]}}
  \rmJ_{k, L'} \big|_{k, L'}\, U_{\td s}
\text{,}
\\[6pt]
  \rmJ_{k, L}^{U \rm new}
& =
  \big( \rmJ_{k,L}^{U \rm old} \big)^\perp
\subseteq
  \rmJ_{k, L}
\text{.}
\end{align*}
The orthogonal complement is taken with respect to the Petersson scalar product
\begin{gather*}
  \big\langle f, g \big\rangle
=
  \int_{\SL{2}{\ZZ} \backslash \HS} \;
  \sum_{\mu \in D} f_{\mu}(\tau) \ov{g_\mu(\tau)} \; \frac{dx \, dy}{y^2}
\text{.}
\end{gather*}

The central result of this sections says that oldforms can be recognized by their restrictions to Jacobi forms of scalar index.
\begin{theorem}
\label{thm:oldforms_by_restriction}
Let ${\td s} \in \Mat{N}{\ZZ}$ and assume that $L = L'[{\td s}]$ for some $L' \in \Mat{N}{\ZZ}$ with even diagonal entries.  Then there is an explicit set $\cS \subset \ZZ^N$ such that the following holds.

For $s \in \cS$, let $l_s$ be the greatest common divisor of the entries of ${\td s} s$.  Suppose that $0 \ne \phi \in \rmJ_{k, L}$.  If $\phi[s] \in \rmJ_{k, L[s] / l_s^2} \big|_{k, L[s] / l_s^2} \, U_{l_s}$ for all $s \in \cS$, then $\phi \in \rmJ_{k, L'} \big|_{k, L'}\, U_{\td s}$.
\end{theorem}

The first assumption of Theorem~\ref{thm:oldforms_by_restriction} can also be checked on restrictions.
\begin{proposition}
\label{prop:sublattice_by_restriction}
Let ${\td s} \in \Mat{N}{\ZZ}$.  Then there is an explicit set $\cS \subset \ZZ^N$ such that the following holds.

For $s \in \cS$, let $l_s$ be the greatest common divisor of the entries of ${\td s} s$.  If $2 l_s^2 \isdiv L[s]$ for all $s \in \cS$, then $L = L'[{\td s}]$ for some $L' \in \Mat{N}{\ZZ}$ with even diagonal.
\end{proposition}
\begin{proof}
There is $0 < l \in \ZZ$ such that $l \ZZ^N \subseteq {\td s} \ZZ^N$.  We show that we can choose
\begin{gather*}
  \cS
=
  \{ l {\td s}^{-1} e_i \,:\, 1 \le i \le N \}
  \cup
  \{ l {\td s}^{-1} (e_i + e_{i + 1}) \,:\, 1 \le i \le N - 1 \}
\text{,}
\end{gather*}
where $e_i$ denotes the $i$-th unit vector.  We have $l_s = l^2$ for all $s \in \cS$, 

Suppose that $2 l_s^2 \isdiv L[s]$ for all $s \in \cS$.  We want to show that $L[{\td s}^{-1}]$ is integral with even diagonal entries.  Since $L[l {\td s}^{-1} s'] = l^2 L[{\td s}^{-1}][s']$ for all $s' \in \ZZ^N$, this follows immediately.
\end{proof}

\begin{proof}[Proof of Theorem~\ref{thm:oldforms_by_restriction}]
Recall the definition of $\cR$ and $\cS^{\rm str}(r)$, $r \in \cR$ from Section~\ref{sec:algorithm}.  We set
\begin{gather*}
  \cS
=
  \bigcup_{r \in \cR} \big\{ l {\td s}^{-1} s' \,:\, s' \in \cS^{\rm str}(r) \big\}
\text{.}
\end{gather*}
Since vectors in $\cS(r)$ are primitive by definition, we have $l_s = l$ for all $l \in \cS$.

Suppose that $\phi[s] \in \rmJ_{k, L[s] / l_s^2} \big|_{k, L[s] / l_s^2}\, U_{l_s}$ for all $s \in \cS$.  We have to show that $\phi \in \rmJ_{k, L'} \big|_{k, L'}\, U_{\td s}$.  This holds if and only if $\cF\cE(\phi)$ is supported on pairs $(n, r)$ for which $({\td s}^{\tr})^{-1} r$ is integral.  Recall the Fourier coefficients $c$ and ${\td c}$ in Section~\ref{sec:modularforms}.  Because $L = s^\tr (L' s)$ and hence ${\td c}(\Delta, r) = {\td c}(\Delta, r + s^\tr (L' s \lambda))$ for $\lambda \in \ZZ^N$ and also because of~(\ref{eq:higherjacobiforms_fourier_relation2}), it suffices to show this for a system of representatives of $\big( \ZZ^N \slashdiv L \ZZ^N \big) \slashdiv \{ \pm 1 \}$. By definition $\cR$ is such a system of representatives.  Henceforth we assume that $r \in \cR$.

We prove the claim by induction on $n$.  Suppose that all $(n', r)$ with $n' < n$ and $c(n', r) \ne 0$ satisfy $({\td s}^{-1})^\tr r \in \ZZ^N$.  Fix $r$ with $c(n, r) \ne 0$ and set ${\td r} = s^\tr r$.  The Fourier coefficient of index $(n, {\td r})$ of $\phi[s]$ equals
\begin{gather*}
  \sum_{r' \in \ZZ^N \,:\, s^\tr r' = {\td r}} c(n, r')
\text{.}
\end{gather*}
Using the induction hypothesis, (\ref{eq:higherjacobiforms_fourier_relation1}), and~(\ref{eq:higherjacobiforms_fourier_relation2}), we can eliminate all contributions of $c(n', r')$, $n' < n$, $r' \in \cR$.  We are then left with the expression
\begin{gather*}
  \sum_{\substack{r'' \in \cR \\ r' \in {\rm neigh}(r'') \,\cup\, {\rm neigh}(-r'') \\ s^\tr r' = {\td r}}} \hspace{-2em}
  c(n, r)
\text{.}
\end{gather*}
The sum only runs over $r' = r$, as can be seen directly from the definition of $\cS^{\rm str}(r)$, and hence it is non-zero.  By assumption, we have $l_s \nisdiv {\td r}$.  We conclude that $l \isdiv s^\tr r = l s^{\prime \tr} ({\td s}^{-1})^\tr r$.  Consequently $s^{\prime \tr} ({\td s}^{-1})^\tr r \in \ZZ$ for all $s' \in \cS^{\rm str}(r)$.  Since $\cS^{\rm str}(r)$ spans $\ZZ^N$, we deduce that all entries of $({\td s}^{-1})^\tr r$ are integral.  This proves the theorem.
\end{proof}

We formulate an algorithm that makes use of Theorem~\ref{thm:oldforms_by_restriction} in order to restrict the necessary computation to newspaces.
\begin{algorithm}
\label{alg:jacobiforms_newforms}
Fix $k \in \ZZ$, $L$ as above, and $0 < B \in \ZZ$.  The following algorithm computes $\cF\cE_{L, B}\big( \rmJ_{k, L}^{U \rm new} \big)$.
\begin{enumerate}[1:]
\item Compute a set $\cR$ of representatives r of $\big( \ZZ^N \slashdiv L \ZZ^N \big) \slashdiv \{ \pm 1 \}$ with minimal norm~$L[r]$ in $r + L \ZZ^N$.
\vspace{0.5ex}
\item Compute a minimal complete set $\cS$ of restriction vectors for $\cR$.  Set $\cS \leftarrow \cS \cup \cS' \cup \cS''$, where $\cS'$ and $\cS''$ are the sets given in Proposition~\ref{prop:sublattice_by_restriction} and Theorem~\ref{thm:oldforms_by_restriction}.
\vspace{0.5ex}
\item \label{it:alg:newforms:adjust_precision}
  Set $B \leftarrow \max\big\{B,\, \max_s \{\dim(M_{k + L[s]},\, \dim \rmM_{k - 1 + L[s]}\} \big\}$.
\vspace{0.5ex}
\item \label{it:alg:newforms:compute_scalar_jacobi_forms}
 For each $s \in \cS$, compute $\cJ^{U \rm old}_{k, L[s], B} := \cF\cE_{L[s], B} \big( \rmJ^{U \rm old}_{k, L[s]} \big) \cap \QQ^{\cI(L[s], B)}$ and $\cJ^{U \rm new}_{k, L[s], B} := \cF\cE_{L[s], B} \big( \rmJ^{U \rm new}_{k, L[s]} \big) \cap \QQ^{\cI(L[s], B)}$.
\vspace{0.5ex}
\item \label{it:alg:newforms:compute_old_dimensions}
  Compute $\dim \rmJ_{k, L}^{U \rm new} = \dim \rmJ_{k, L} - \dim \rmJ_{k, L}^{U \rm old}$ by recursively applying this algorithm.
\vspace{0.5ex}
\item Compute $\cJ_{k, L, B} := \big\{ (v[s])_{s \in \cS} \,:\, v \in \QQ^{\cI(L, B)} \big\} \subseteq \bigoplus_{s \in \cS} \QQ^{\cI(L[s], B)}$.
\vspace{0.5ex}
\item \label{it:alg:newforms:check}
      Set
      \begin{gather*}
        \cJ^{U \rm new}_{k, L}
      \longleftarrow
        \cJ_{k, L}
      \cap
        \bigoplus_{(l_s)_s} 
        \Big(
          \bigoplus_{s \in \cS'} \cJ^{U \rm new}_{k, L[s] / l_s^2, B} \big|\, U_{l_s}
        \Big)
      \text{,}
      \end{gather*}
      where the outer sum runs over tuples $(l_s)_{s \in \cS}$ of positive integers which satisfy $\gcd((l_s)_{s \in \cS}) = 1$ and $l_s^2 \isdiv L[s]$ for all $s \in \cS$.  If $\dim\, \cJ^{U \rm new}_{k, L} \ne \dim \rmJ_{k, L}^{U \rm new}$, enlarge $\cS$ by the shortest possible vector with respect to $q_\cL$.  Continue with Step~\ref{it:alg:newforms:adjust_precision}.
\vspace{0.5ex}
\item Compute the inverse image of $\cJ^{\rm restr}_{k, L}$ under $v \mapsto (v[s])_{s \in \cS}$.  It equals $\cF\cE_B\big(\rmJ^{U \rm new}_{k, L}\big)$.  Truncate it to the required precision.
\end{enumerate}
\end{algorithm}
\begin{proof}[Proof of correctness and termination]
Correctness of Algorithm~\ref{alg:jacobiforms_newforms} follows from Theorem~\ref{thm:oldforms_by_restriction}.  Because we have
\begin{gather*}
  \cF\cE \big( \rmJ_{k, L}^{U \rm old} \big)
  \oplus
  \bigoplus_{(l_s)_{s \in \cS}} 
  \Big(
    \bigoplus_{s \in \cS'} \cJ^{U \rm new}_{k, L[s] / l_s^2, B} \big|\, U_{l_s}
  \Big)
=
  \bigoplus_{s \in \cS} \cJ_{k, L[s], B}
\text{,}
\end{gather*}
Algorithm~\ref{alg:jacobiforms_newforms} terminates.
\end{proof}

We finish this section with some comments on Algorithm~\ref{alg:jacobiforms_newforms}.  First of all, it allows for a significant reduction of memory consumption, if the $\cJ_{k, L[s], B}$ is large, because it allows us to split these spaces.  In contrast to algorithms for elliptic modular forms that make use of the newform decomposition, it does not necessarily reduce the time consumption for the following reason.  While in the case of elliptic modular forms one can show that the corresponding decomposition into newspaces is a direct sum, this is not the case for the decomposition of $\rmJ_{k, L}$ arising from the $U_s$.  As we have seen, not even formula (\ref{eq:newform_decomposition}) corresponds to such a direct sum decomposition if $N \ne 1$.  If an analog decomposition could be found, we would be able to compute the corresponding dimensions of newspaces recursively.  Thus we would not have to compute oldspaces in the course of Algorithm~\ref{alg:jacobiforms_newforms}, saving, indeed, much time.


%% file: implementation.tex
\section{Implementing Algorithm~\ref{alg:jacobiforms}}
\label{sec:implementation}

This section contains instructions on how to use the implementation of Algorithm~\ref{alg:jacobiforms} that is provided by the author.  We will also discuss the algorithm's and implementation's bottleneck.

At the time this paper was written, the implementation of Jacobi forms was not yet integrated in Sage.  Pull the ticket \#16448 for the Sage development server.
\begin{lstlisting}[basicstyle=\ttfamily\small]
sage -dev checkout --ticket 16448
\end{lstlisting}

Start Sage, either in the terminal or as a notebook, and type:
\begin{lstlisting}[basicstyle=\ttfamily\small]
from sage.modular.jacobi.all import *
k = 9
L = QuadraticForm(matrix(2, [2,1,1,2]))
B = 10
jforms = higherrank_jacobi_forms(k, L, B)
\end{lstlisting}
This computes the module of Jacobi forms of weight $k = 10$ and index $L = \left(\begin{smallmatrix}2 & 1 \\ 1 & 2 \end{smallmatrix}\right)$, where the underlying Fourier expansions $\sum_{n, r} c(n, r)\, q^n \zeta^r$ are truncated at $n < B = 10$.  The return value is a list of dictionaries, which represent Fourier expansions of Jacobi forms.  In order to access the Fourier coefficient at $n = 2$ and $r = (1, 1)$ of the first basis element ($\rmJ_{k,L}$ has dimension~$2$), type:
\begin{lstlisting}[basicstyle=\ttfamily\small]
n = 2; r = (1, 1)
jforms[0][(n,r)]
\end{lstlisting}

This works, since $r = (1,1)$ is a reduced vector.  For general~$r$, we have to invoke index reduction, to find a corresponding index of the dictionary.
\begin{lstlisting}[basicstyle=\ttfamily\small]
n = 10; r = (3, -2)
L_adj = QuadraticForm(2 * L.matrix().adjoint())
(r_classes, _) = higherrank_jacobi_r_classes(L)
L_span = L.matrix().row_module()
((nred, rred), s) = higherrank_jacobi_reduce_fe_index((n,r),
                                    L, r_classes, L_adj, L_span)
s**k * jforms[0][(nred, rred)]
\end{lstlisting}
It is vital to multiply the Fourier coefficient at the reduced index by the $k$-th power of the sign~$s$ that is returned by the index reduction function.  This sign stems from relation~\eqref{eq:higherjacobiforms_fourier_relation2}.

Our major goal is to compute Fourier expansions of vector valued modular forms.  Such functionality is provided via Jacobi forms and the theta decomposition, but there is a separate interface.  Given a Jacobi form as above, we find the associated vector valued modular form as follows:
\begin{lstlisting}[basicstyle=\ttfamily\small]
theta_decomposition(jforms[0], L, r_classes)
\end{lstlisting}
The result is a dictionary of dictionaries, that represents the Fourier expansion of a vector valued modular form.  The above call yields
\begin{lstlisting}[basicstyle=\ttfamily\small]
{(0, 0):
 {0: 0, 1: 0, 2: 0, 3: 0, 4: 0, 5: 0, 6: 0, 7: 0, 8: 0, 9: 0},
 (0, 1):
 { 2/3: -1,     5/3: 16,    8/3: -104,  11/3: 320, 14/3: -260,
  17/3: -1248, 20/3: 3712, 23/3: -1664, 26/3: -6890},
 (0, 2):
  { 2/3:  1,     5/3: -16,    8/3: 104,  11/3: -320, 14/3: 260,
   17/3: 1248, 20/3: -3712, 23/3: 1664, 26/3: 6890}}
\end{lstlisting}
The outer keys are tuples of integers, which yield representatives of the discriminant form $\ZZ^N \slashdiv L \ZZ^N$ associated with~$L$.  In the present case, they correspond to components of a vector valued modular form with respect to the basis $\frake_{(0,0)}$, $\frake_{(1,0)}$, $\frake_{(2,0)}$ of $\CC\big[ \ZZ^N \slashdiv \left(\begin{smallmatrix} 2 & 1 \\ 1 & 2 \end{smallmatrix}\right) \ZZ^N  \big]$.  The inner keys are rationals, which correspond to exponents of~$q$.  Summarizing, the above Python dictionary corresponds to the following $q$\nbd expansion:
\begin{multline*}
  \big( \frake_{(0,2)} - \frake_{(0,1)} \big)
  \big(
  q^{\frac{2}{3}}
  - 16 q^{\frac{5}{3}}
  + 104 q^{\frac{8}{3}}
  - 320 q^{\frac{11}{3}}
  + 260 q^{\frac{14}{3}}
  + 1248 q^{\frac{17}{3}}
\\
  - 3712 q^{\frac{20}{3}}
  + 1664 q^{\frac{23}{3}}
  + 6890 q^{\frac{26}{3}}
  \big)
\text{.}
\end{multline*}

The theta decomposition is used throughout the implementation to compute vector valued modular forms for the conjugate~$\ov{\rho}_L$ of the Weil representation attached to~$L$.  This method of computation requires that $L$ is positive definite.  The next lines of code compute a lattice that is stably equivalent to~$(-2)$.
\begin{lstlisting}[basicstyle=\ttfamily\small]
L = QuadraticForm(diagonal_matrix([-2]))
Lpos = positive_definite_quadratic_form(L)
\end{lstlisting}
Given this lattice, we can compute vector valued modular forms as follows.
\begin{lstlisting}[basicstyle=\ttfamily\small]
k = 5/2
B = 5
vector_valued_modular_forms(k, Lpos, B)
\end{lstlisting}

\subsection{Bottlenecks, parallelisation, and extensions}

There are two major bottlenecks of our implementation that can both be mitigated by parallelisation.  When computing Jacobi forms for large lattices of small discriminant the most time consuming step is to enumerate short vectors in the lattice.  Sage calls Pari~\cite{pari} via pexpect, and then sorts the corresponding vectors by length.  Obviously, this is very slow.  It also provides unnecessary data, the vectors' lengths, in some cases.  A more efficient implementation of Algorithm~\ref{alg:jacobiforms} should be based on optimized algorithms for this task.  Such implementation are provided in recent versions of fpLLL~\cite{fpLLL} and at~\cite{sageticket-enumeration}, but they have not yet been wrapped in Sage.  It is also possible to parallelize enumeration of short vectors~\cite{DS10}.

When computing Jacobi forms for lattices of large discriminant up to high precisions, linear algebra can become a bottleneck.  Optimization and parallelization of this step is of general interest, but it is also possible to split up the original problem by using newform theories as indicated in Section~\ref{sec:newforms_algorithm}.  We have already made clear, however, that currently there no completely satisfactory theory of this kind.  In particular, it would be important to have explicit dimension formulas for spaces of newforms.

The provided implementation of Algorithm~\ref{alg:jacobiforms} does not expose one important feature of our algorithm.  Since restricting formal Foourier expansions of Jacobi indices $L$ to $L[s]$ for some $s \in \ZZ^N$ is independent of the weight modulo~$2$, it suffices to calculate the corresponding matrix once.  Note that the dependence modulo~$2$ can not be eliminated if only reduced indices are considered.  A cached matrix can be reused for all weights, as long as the set of restriction vectors~$\cS$ does not need to be enlarged and the precision $B$ does not need to be increased.  For an example on how to cache these matrices, the reader is referred to the files at~\cite{raumhomepage} which are related to Section~\ref{sec:data}.

Finally, we consider restrictions to Jacobi forms of higher rank index.  It is not necessary to restrict the Jacobi forms that we wish to compute to $s z$, $z \in \CC$ for vectors $s \in \ZZ^N$.  One can equally well restrict them to $s z$, $z \in \CC^{N'}$ for any sufficiently large set of matrices $s \in \Mat{N', N}{\ZZ}$, $N' < N$.  This can be seen directly by applying Lemma~\ref{la:vanishing_s_of_higher_jacobi_forms} to the obvious generalization of Proposition~\ref{prop:modularity_s_of_higher_jacobi_forms}.  In order to compute Fourier expansions of Jacobi forms of index $L[s]$, $s \in \Mat{N', N}{\ZZ}$ one has to, eventually, restrict to scalar indices.  However, it is possible to break up the problem into smaller pieces, which can be treated in parallel.  Also, since restriction to Jacobi forms of higher index pertains more information, it is conceivable that the working precision $B$ can be decrease.  For lattices of large size this results in a significant reduction of the number of short vectors that have to be enumerated.

The previous observation gains even more importance in the light of the following speculation.  In some cases there is a geometric interpretation of embeddings $\cL' \hookrightarrow \cL$.  Such an interpretation possibly allows for computing the restriction map on Fourier expansions more efficiently than by bare enumeration of vectors in $\cL$.  This kind of tricks could be helpful when computing Jacobi forms for, say, rescaled Leech lattices.

\subsection{Input data to~\cite{GKR11}}
\label{ssec:input_data}

In~\cite{GKR11}, we used the presented implementation to compute Fourier expansions of weakly holomorphic vector valued modular forms that serve as input data to the multiplicative Borcherds lift.  Input data to Borcherds products for the orthogonal group $\cL = \cU \oplus \cU \oplus \cL_0(-1)$, where $\cU$ denotes the hyperbolic plane, corresponds to Jacobi forms of positive index~$\cL_0$.  This is a relatively easy application of the presented algorithm, since $\cL_0$ is positive definite, and so we do not need to apply Algorithm~\ref{alg:positive_discriminant_forms}.  Because weakly holomorphic vector valued modular forms are connected to vector valued modular forms via multiplication of sufficiently high powers of the modular form~$\Delta$,  one can compute them without major difficulty.  We provide facilities to compute weakly holomorphic modular forms of type $\rho_{\cL_0(-1)}$.  Such forms are uniquely defined by their principal part, if the weight is not positive.
\begin{gather*}
  \sum_{\substack{\mu \in \disc\, \cL_0 \\ \QQ \ni m < 0}} a(m, \mu) \, q^m \, \frake_\mu
\text{.}
\end{gather*}

We illustrate how to use our implementation.
\begin{lstlisting}[basicstyle=\ttfamily\small]
from sage.modular.jacobi.all import *
k = -1
L = QuadraticForm(matrix(2, [2,1,1,2]))
order = 1
B = 5
vector_valued_modular_forms_weakly_holomorphic(k, L, order, B)
\end{lstlisting}
In case, we wish to find a weakly holomorphic vector valued modular form with prescribed principal part, we can do so as follows.
\begin{lstlisting}[basicstyle=\ttfamily\small]
pp = {(0,0): {-1: 1}}
vector_valued_modular_forms_weakly_holomorphic_\\
                            with_principal_part(k, L, pp, B)
\end{lstlisting}
The double slash indicates line continuation.  The principal part of a vector valued modular form is its Fourier expansion truncated at~$q^0$.  In particular, the above principal part is~$q^{-1} \frake_{(0,0)}$.   The function will fail, raising ValueError, if the desired vector valued modular form does not exist.


%% file: data.tex
\section{Vector valued Jacobi forms for lattices of rank~$1$ and~$2$}
\label{sec:data}

Vector valued elliptic modular forms associated to cyclic discriminant groups are of particular interest, since they are the most basic case.  Among the cyclic ones, the discriminant groups $D_m := \big(\ZZ \slashdiv 2m \ZZ, v \mapsto v^2 \slashdiv 2m\big)$, $m > 0$ are the simplest.  They come from positive definite lattices $(2 m)$.  Consequently, the corresponding vector valued modular forms of type~${\check \rho}_D$ can be obtained as components of the theta decomposition of holomorphic Jacobi forms of index~$m$.  An algorithm to compute their Fourier expansions has been known for a long time.  A first version was given in~\cite{EZ85}, and a much more refined version, which is also used in Sage~\cite{sageticket-jacobi}, was presented in~\cite{Sk84}.

From a structural point of view, the discriminant forms $D_{m} := \big(\ZZ \slashdiv 2m \ZZ, v \mapsto v^2 \slashdiv 2m\big)$, $m < 0$ are equally simple, but it is harder to compute Fourier coefficients of associated vector valued modular forms.  One can employ a theta decomposition to relate such vector valued modular forms to skew-holomorphic Jacobi forms as defined in~\cite{Sk90}.  Apart from Eisenstein series and theta series, no construction of skew-holomorophic forms is known.  The major complication when dealing with skew-holomorphic Jacobi forms is that the product of two of them is no longer a skew-holomorphic Jacobi form.  Thus the Fourier expansions of skew-holomorphic Jacoobi forms are much harder to compute than those of holomorphic Jacobi forms.  Algorithm~\ref{alg:positive_discriminant_forms} and~\ref{alg:jacobiforms} provide us with a relatively easy way to compute the Fourier coefficients of vector valued modular forms of type~${\check \rho}_{D_m}$.  The Jacobi indices that arise during these computations have size $7 \times 7$.  As explained in Section~\ref{sec:algorithm}, enumeration of short vectors in such lattices is time consuming and Sage does not come with an optimized implementation.  For this reason we have precomputed several cases that are possibly interesting for further applications.

Using hardware provided by the Max Planck Institute in Bonn, we have computed Fourier expansions up to precision~$40$ of vector valued modular forms of type~${\check \rho}_{D_m}$ for all $0 > m \ge -20$.  A short excerpt is shown in Table~\ref{tab:fourier_expansions_skew_m1} and~\ref{tab:fourier_expansions_skew_m5}.  The complete data is available at~\cite{raumhomepage}.  We illustrate how to access and use it.  Recall that the elements of $M_k({\check \rho}_D)$ have a Fourier expansion of the form
\begin{gather*}
  f (\tau) = \sum_{\substack{\mu \in D \\ 0 \le n \in \QQ}} a_f(m, \mu) \, q^m \,\frake_\mu
\text{.}
\end{gather*}
The following code only requires Sage~\cite{sage} to be installed.  We consider the case $m = 2$, $k = \frac{5}{2}$, and we want to know the value of $a_f(0, 4)$, where $f$ is the basis element of $\rmM_{\frac{5}{2}}({\check \rho}_{D_{-2}})$ that we have computed.  We simply type
\begin{lstlisting}[basicstyle=\ttfamily\small]
fs = load(vector_valued_fourier_expansion__2k_5__m_-2__.sobj')
(fs[0])[0][4]
\end{lstlisting}
This gives us~$-312$.  The code does the following: We load a list \verb!fs! of basis elements for the space $M_{\frac{5}{2}}\big( {\check \rho}_{D_{-2}} \big)$.  In our case this list has length~$1$.  The elements of this list are dictionaries whose keys correspond to elements in $\ZZ \slashdiv 2 |m| \ZZ$, which label the components of modular forms of type~${\check \rho}_{D_{-m}}$.  They are integers $0 \le r \le |m|$.  Note that the Fourier expansion of any element of $M_k\big( {\check \rho}_{D_{m}} \big)$ is supported on $\frake_{\mu} + \frake_{-\mu}$, $\mu \in D$ in the notation of Section~\ref{sec:modularforms}.  Hence the available values suffice to reconstruct all Fourier coefficients $a_f(\mu, n)$ for all $\mu$.

We have also precomputed Fourier expansion of vector valued modular forms for negative definite lattices~$\cL$ of rank~$2$:
\begin{gather*}
  \begin{pmatrix} 2 & 1 \\ 1 & 2 \end{pmatrix}
\text{,}\quad
  \begin{pmatrix} 2 & 0 \\ 0 & 2 \end{pmatrix}
\text{,}\quad
  \begin{pmatrix} 2 & 1 \\ 1 & 4 \end{pmatrix}
\text{,}\quad
  \begin{pmatrix} 2 & 0 \\ 0 & 4 \end{pmatrix}
\text{.}
\end{gather*}
These expansions are related to~\cite{GKR11}, where they serve as input data to an algorithm computing Borcherds products.  The corresponding weakly holomorphic vector valued modular forms can be accessed as was described in Section~\ref{ssec:input_data}.  For the given indices all elements of $\rmM_{-1}(\rho_{\cL(-1)})$ with pole order at infinity less than or equal to~$5$ are provided.

\begin{table}[h]
\begin{tabular}{crrrrrr}
\toprule
$\mu$ & \multicolumn{6}{c}{$m$} \\
\cmidrule{2-7}
$0$ & 
  $0$ & $1$ & $2$ & $3$ & $4$ & $5$ \\[4pt]
$1$ &
  $\frac{1}{4}$ & $\frac{5}{4}$ & $\frac{9}{4}$ & $\frac{13}{4}$ &
  $\frac{17}{4}$ & $\frac{21}{4}$ \\[6pt]
\midrule
$a_{\frac{5}{2}}(m, \mu)$ &
  $1$ & $-70$ & $-120$ & $-240$ & $-550$ & $-528$ \\
                         &
  $-10$ & $-48$ & $-250$ & $-240$ & $-480$ & $-480$ \\[6pt]
$a_{\frac{9}{2}}(m, \mu)$ &
  $1$ & $242$ & $2640$ & $11040$ & $30962$ & $65760$ \\
                         &
  $2$ & $480$ & $4322$ & $13920$ & $39360$ & $73920$ \\[6pt]
$a^{(1)}_{\frac{13}{2}}(m, \mu)$ &
$28$ & $0$ & $-433680$ & $-4736160$ & $-21626280$ & $-74216064$ \\
                         &
  $-195$ & $-58344$ & $-933595$ & $-6583560$ & $-30703920$ & $-96379920$ \\[6pt]
$a^{(2)}_{\frac{13}{2}}(m, \mu)$ &
  $0$ & $56$ & $240$ & $-1440$ & $704$ & $-960$ \\
                                &
  $-1$ & $-120$ & $-9$ & $1320$ & $240$ & $-5040$ \\
\bottomrule \\
\end{tabular}
\caption{Fourier coefficients $a_k(m, \mu)$ of basis elements of the spaces of vector valued elliptic modular forms of weight~$k$ and type ${\check \rho}_{(-2)}$.}
\label{tab:fourier_expansions_skew_m1}
\end{table}

\begin{table}
\begin{tabular}{crrrrrrrr}
\toprule
$\mu$ & \multicolumn{8}{c}{$m$} \\
\cmidrule{2-9}
$0$ & 
  $0$ & $1$ & $2$ & $3$ & $4$ & $5$ & $6$ & $7$ \\[4pt]
$1$ &
  $\frac{1}{20}$ & $\frac{21}{20}$ & $\frac{41}{20}$ & $\frac{61}{20}$ &
  $\frac{81}{20}$ & $\frac{101}{20}$ & $\frac{121}{20}$ & $\frac{141}{20}$ \\[4pt]
$2$ &
  $\frac{1}{5}$ & $\frac{6}{5}$ & $\frac{11}{5}$ & $\frac{16}{5}$ &
  $\frac{21}{5}$ & $\frac{26}{5}$ & $\frac{31}{5}$ & $\frac{36}{5}$ \\[4pt]
$3$ &
  $\frac{9}{20}$ & $\frac{29}{20}$ & $\frac{49}{20}$ & $\frac{69}{20}$ &
  $\frac{89}{20}$ & $\frac{109}{20}$ & $\frac{129}{20}$ & $\frac{149}{20}$ \\[4pt]
$4$ &
  $\frac{4}{5}$ & $\frac{9}{5}$ & $\frac{14}{5}$ & $\frac{19}{5}$ &
  $\frac{24}{5}$ & $\frac{29}{5}$ & $\frac{34}{5}$ & $\frac{39}{5}$ \\[4pt]
$5$ &
  $\frac{1}{4}$ & $\frac{5}{4}$ & $\frac{9}{4}$ & $\frac{13}{4}$ &
  $\frac{17}{4}$ & $\frac{21}{4}$ & $\frac{25}{4}$ & $\frac{29}{4}$ \\[6pt]
\midrule
$a^{(1)}_{\frac{5}{2}}(m, \mu)$ &
  $5$ & $0$ & $-120$ & $-960$ & $-840$ & $-2270$ & $-960$ & $-2160$ \\[4pt]
              &
  $-7$ & $60$ & $-420$ & $-660$ & $-1147$ & $-1080$ & $-2647$ & $-2220$ \\[4pt]
              &
  $17$ & $-240$ & $-120$ & $-763$ & $-1320$ & $-1560$ & $-1440$ & $-2563$ \\[4pt]
              &
  $-43$ & $-240$ & $-643$ & $-300$ & $-1200$ & $-1500$ & $-2460$ & $-1260$ \\[4pt]
              &
  $-187$ & $-367$ & $-360$ & $-1080$ & $-840$ & $-1320$ & $-1920$ & $-2400$ \\[4pt]
              &
  $-60$ & $-230$ & $-540$ & $-840$ & $-120$ & $-1560$ & $-1140$ & $-3360$ \\[6pt]
$a^{(2)}_{\frac{5}{2}}(m, \mu)$ &
  $0$ & $20$ & $40$ & $-80$ & $0$ & $-60$ & $120$ & $-80$ \\[4pt]
              &
  $-1$ & $30$ & $-10$ & $-30$ & $29$ & $-40$ & $-21$ & $-110$ \\[4pt]
              &
  $6$ & $-20$ & $40$ & $16$ & $-60$ & $-80$ & $80$ & $16$ \\[4pt]
              &
  $1$ & $-20$ & $1$ & $50$ & $0$ & $-50$ & $-30$ & $70$ \\[4pt]
              &
  $-16$ & $-6$ & $20$ & $-40$ & $80$ & $40$ & $40$ & $0$ \\[4pt]
              &
  $-10$ & $10$ & $-50$ & $-20$ & $140$ & $20$ & $50$ & $-80$ \\
\bottomrule \\
\end{tabular}
\caption{Fourier coefficients $a_k(m, \mu)$ of basis elements of the spaces of vector valued elliptic modular forms of weight~$k$ and type ${\check \rho}_{(-10)}$.}
\label{tab:fourier_expansions_skew_m5}
\end{table}


%% file: specialdivisors.tex
\section{Special divisors on modular varieties of othogonal type}
\label{sec:specialdivisors}

The data provided in the previous section can be used to deduce linear equivalences of special divisors on modular varieties of othogonal type.  Let $\cL$ be a lattice of signature $(2, n)$, $n \ge 2$.  The orthogonal group~$\Orth{}(\cL)$ of $\cL$ consists of all linear transformation of $\cL \otimes \QQ$ which preserve $q_\cL$ and $\cL$.  The discriminant kernel $\Orth{}(\cL)[\disc\,\cL] \subseteq \Orth{}(\cL)$ is the kernel of the action of $\Orth{}(\cL)$ on $\disc\, \cL = \cL^\# \slashdiv \cL$.  A modular variety of orthogonal type is of the form $\Orth{}(\cL)[\disc\, \cL] \backslash \Orth{2, n} \slashdiv \Orth{2} \times \Orth{n}$.  We will give a model for this variety that allows us to easily define special divisors.  The reader is referred to~\cite{Br02} for details on the construction that we explain. Set
\begin{gather*}
  \cD_e
=
  \big\{ z \in \cL \otimes \CC \,:\, q_{\cL}(z) = 0,\, \langle z, \ov{z} \rangle_{\cL} > 0 \big\}
\text{,}
\end{gather*}
where $\ov{z}$ means the complex conjugate in the second tensor component.  It is well-known that $\cD_e$ has two connected components.  Choose one and denote it by $\cD_e^+$.  The projectivisation $\cD = \cD_e^+ \slashdiv ( \CC \setminus \{0\} )$ is a homogenous space for $\Orth{2, n}$.  We write $X_\cL$ for  $\Orth{}(\cL)[\disc\, \cL] \backslash \cD$.  The linear extension of $\langle \,\cdot, \cdot\,\rangle_\cL$ to $\cL \otimes \CC$ allows us to define divisors in $\cD_e$ and $\cD$ as follows.  Given $\lambda \in \cL^\#$, set
\begin{gather*}
  Z(\lambda)
=
  \{ z \in \cD \,:\, \langle z, \lambda \rangle_{\cL} = 0 \}
\text{.}
\end{gather*}
This is an analytic divisor on $\cD$.  For $\mu \in \disc\, \cL$ and $0 < m \in \QQ$, the special divisor of discriminant $(m, \mu)$ is defined as
\begin{gather}
  Z(m, \mu)
=
  \sum_{\substack{ \lambda \in \cL + \mu \\ q_{\cL}(\lambda) = -m}} Z(\lambda)
\text{.}
\end{gather}
Write $\ZZ\big[ Z(m, \mu) \big]$ for the free module with basis $Z(m, \mu)$, $0 < m \in \QQ$$, \mu \in \disc\, \cL$.  The special Picard group ${\rm Pic}_{\rm sp}(X_\cL)$ of $X_\cL$ is
\begin{gather*}
  \ZZ\big[ Z(m, \mu) \big]
  \slashdiv
  \big\{ {\rm div}(f) \,:\, f \text{ meromorphic on $X_\cL$}, {\rm div}(f) \subset \ZZ\big[ Z(m, \mu) \big] \big\}
\text{.}
\end{gather*}

Assume that $\cL = \cU \oplus \cU \oplus \cL'(-1)$, where $\cU$ is the hyperbolic plane.  In this case its is known by~\cite{Br02, BF10} that linear equivalences of special divisors on~$X_\cL$ are given by relations of coefficients of vector valued modular forms.  Linear equivalences of special divisors are the same as relations in~${\rm Pic}_{\rm sp}(X_\cL)$.  Specifically, let $(f_i)_{i}$ be a basis of $M_{\frac{2 + n}{2}}({\check \rho}_{\cL})$ with Fourier expansion
\begin{gather*}
  f_i(\tau) = \sum_{\mu \in \disc\, \cL} \sum_{m \in \QQ} a_i(m, \mu) \, \frake_\mu
\text{.}
\end{gather*}
Given a finite sum $Z = \sum_{\mu \in \disc\, \cL} \sum_{0 < m \in \QQ} b(m, \mu) Z(-m, \mu)$, we have $0 \eqPic Z$ in ${\rm Pic}_{\rm sp}(X_\cL)$ if and only if
\begin{gather}
\label{eq:special_divisor_relation}
  \sum_{\substack{\mu \in \disc\, \cL' \\ 0 < m \in \QQ}}
  a_i(m, \mu)\, b(m, \mu)
=
  0
\end{gather}
for all $1 \le i \le \dim(M_{\frac{2 + n}{2}}({\check \rho}_{\cL}))$.

In the case $\cL' = (m)$, we say that $X_\cL$ is a paramodular variety of genus~$2$ and level~$m$.  The computations in Section~\ref{sec:data} can be used to get precise relations of special divisors on these $X_\cL$.
\begin{corollary}
\label{cor:linear_equivalence_m5}
The rational Picard group of degree~$0$ divisors on the toroidal compactification of~$X_{\cU \oplus \cU \oplus (-10)}$ has rank~$1$, and is generated by $Z(1,0)$.  The following relations generate all relations that hold for Heegner divisors $Z(m, \mu)$ with $m < 2$.
\begin{alignat*}{3}
  Z(\tfrac{1}{4}, 5) &= \frac{-1}{2} Z(1,0)
\text{,}
\quad&
  Z(\tfrac{4}{5}, 4) &= \frac{-4}{5} Z(1,0)
\text{,}
\quad&
  Z(\tfrac{9}{20}, 3) &= \frac{1}{20} Z(1,0)
\text{,}
\\
  Z(\tfrac{1}{5}, 2) &= \frac{3}{10} Z(1,0)
\text{,}
\quad&
  Z(\tfrac{1}{20}, 1) &= \frac{-1}{20} Z(1,0)
\text{.}
\end{alignat*}
\end{corollary}
\begin{proof}
The statement on the rank of the Picard group is proved by Bruinier~\cite{Br12}.

To prove the relations, we have to consider vector valued modular cusp forms of weight $\frac{5}{2}$, whose coefficients are given in Table~\ref{tab:fourier_expansions_skew_m5}.  It is easy to check that relation~(\ref{eq:special_divisor_relation}) is satisfied in all cases, and that they span the space of all relations over~$\QQ$.
\end{proof}
More such relations can be computed using the data presented in Section~\ref{sec:data}.  The question of how to compute modular forms with divisors as in the above corollaries was considered in~\cite{GKR11}.


%% file: manuscript.bbl
\newcommand{\etalchar}[1]{$^{#1}$}
\providecommand{\bysame}{\leavevmode\hbox to3em{\hrulefill}\thinspace}
\providecommand{\MR}{\relax\ifhmode\unskip\space\fi MR }
\providecommand{\MRhref}[2]{%
  \href{http://www.ams.org/mathscinet-getitem?mr=#1}{#2}
}
\providecommand{\href}[2]{#2}
\begin{thebibliography}{Rau14b}

\bibitem[BF10]{BF10}
J.~Bruinier and J.~Funke, \emph{On the injectivity of the {K}udla-{M}illson
  lift and surjectivity of the {B}orcherds lift}, Moonshine: the first quarter
  century and beyond, London Math. Soc. Lecture Note Ser., vol. 372, Cambridge
  Univ. Press, Cambridge, 2010, pp.~12--39.

\bibitem[BK01]{BK01}
J.~Bruinier and M.~Kuss, \emph{Eisenstein series attached to lattices and
  modular forms on orthogonal groups}, Manuscripta Math. \textbf{106} (2001),
  no.~4, 443--459.

\bibitem[Bor98]{Bo98}
R.~Borcherds, \emph{Automorphic forms with singularities on {G}rassmannians},
  Invent. Math. \textbf{132} (1998), no.~3, 491--562.

\bibitem[Bor00]{Bo00b}
\bysame, \emph{Reflection groups of {L}orentzian lattices}, Duke Math. J.
  \textbf{104} (2000), no.~2, 319--366.

\bibitem[Bru02]{Br02}
J.~Bruinier, \emph{Borcherds products on {O}(2, {$l$}) and {C}hern classes of
  {H}eegner divisors}, Lecture Notes in Mathematics, vol. 1780,
  Springer-Verlag, Berlin, 2002.

\bibitem[Bru12]{Br12}
\bysame, \emph{On the converse theorem for {B}orcherds products}, ArXiv:
  1210.4821v1, 2012.

\bibitem[CD12]{CD12}
M.~Cheng and J.~Duncan, \emph{The {L}argest {M}athieu {G}roup and ({M}ock)
  {A}utomorphic {F}orms}, ArXiv: 1201.4140v1, 2012.

\bibitem[CG08]{CG08}
F.~Clery and V.~Gritsenko, \emph{The {S}iegel modular forms of genus {$2$} with
  the simplest divisor}, arXiv:0812.3962 [math.NT], 2008.

\bibitem[CPS14]{fpLLL}
D.~Cadé, X.~Pujol, and D.~Stehlé, \emph{fplll-4.0, a floating-point lll
  implementation}, 2014, \url{http://perso.ens-lyon.fr/damien.stehle/fplll}.

\bibitem[Das10]{Das10}
S.~Das, \emph{{S}ome {P}roblems on {J}acobi {F}orms}, Ph.D. thesis, Homi Bhabha
  National Institute, India, 2010.

\bibitem[DS10]{DS10}
\"{O}. Dagdelen and M.~Schneider, \emph{Parallel enumeration of shortest
  lattice vectors}, Proceedings of the 16th international Euro-Par conference
  on Parallel processing: Part II, Euro-Par'10, Springer-Verlag, 2010,
  pp.~211--222.

\bibitem[EOT11]{EOT10}
T.~Eguchi, H.~Ooguri, and Y.~Tachikawa, \emph{{Notes on the $K3$ Surface and
  the Mathieu group $M_{24}$}}, Exper.Math. \textbf{20} (2011), 91--96.

\bibitem[EZ85]{EZ85}
M.~Eichler and D.~Zagier, \emph{The {T}heory of {J}acobi {F}orms},
  Birkh\"auser, Boston, 1985.

\bibitem[Fis87]{Fi87}
J.~Fischer, \emph{An approach to the {S}elberg trace formula via the {S}elberg
  zeta-function}, Lecture Notes in Mathematics, vol. 1253, Springer-Verlag,
  Berlin, 1987.

\bibitem[Fre91]{Fr91}
E.~Freitag, \emph{Singular modular forms and theta relations}, Lecture Notes in
  Mathematics, vol. 1487, Springer-Verlag, Berlin, 1991.

\bibitem[GKR11]{GKR11}
D.~Gehre, J.~Kreuzer, and M.~Raum, \emph{Computing borcherds products},
  Preprint, 2011.

\bibitem[Gri94]{Gr94}
V.~Gritsenko, \emph{Induction in the theory of zeta functions}, Algebra i
  Analiz \textbf{6} (1994), no.~1, 3--63.

\bibitem[Igu62]{Ig84}
J.~Igusa, \emph{On {S}iegel modular forms of genus two}, Amer. J. Math.
  \textbf{84} (1962), 175--200.

\bibitem[JS05]{JS05}
D.~Jatkar and A.~Sen, \emph{Dyon spectrum in {CHL} models}, Journal of High
  Energy Physics \textbf{2006} (2005), no.~04, 38.

\bibitem[Kri90]{Kr90}
A.~Krieg, \emph{Hecke algebras.}, Mem. Am. Math. Soc. \textbf{435} (1990), 158.

\bibitem[Nik79]{Ni79}
V.~Nikulin, \emph{Integer symmetric bilinear forms and some of their geometric
  applications}, Izv. Akad. Nauk SSSR Ser. Mat. \textbf{43} (1979), no.~1,
  111--177, 238.

\bibitem[Par03]{pari}
\emph{{PARI/GP}, {V}ersion 2.5.1}, Bordeaux, 2003,
  \url{http://pari.math.u-bordeaux.fr/}.

\bibitem[PY07]{PY07}
C.~Poor and D.~Yuen, \emph{Computations of spaces of {S}iegel modular cusp
  forms}, J. Math. Soc. Japan \textbf{59} (2007), no.~1, 185--222.

\bibitem[Rau12a]{raumhomepage}
M.~Raum, \emph{Homepage}, 2012, \url{http://www.raum-brothers.eu/martin/}.

\bibitem[Rau12b]{Ra12}
\bysame, \emph{Powers of {$M_{24}$}-twisted {S}iegel product expansions are
  modular}, ArXiv: 1208.3453, 2012.

\bibitem[Rau14a]{sageticket-jacobi}
\bysame, \emph{Genus 1 jacobi forms}, 2014, Sage Trac Server \#16448.

\bibitem[Rau14b]{sageticket-enumeration}
\bysame, \emph{Reimplement short vector enumeration}, 2014, Sage Trac Server
  \#15758.

\bibitem[S{\etalchar{+}}14]{sage}
W.~Stein et~al., \emph{{S}age {M}athematics {S}oftware ({V}ersion 6.2)}, The
  Sage~Development Team, 2014, \url{http://www.sagemath.org}.

\bibitem[Sch04]{Sch04}
N.~Scheithauer, \emph{Generalized {K}ac-{M}oody algebras, automorphic forms and
  {C}onway's group. {I}}, Adv. Math. \textbf{183} (2004), no.~2, 240--270.

\bibitem[Sch08]{Sch08}
\bysame, \emph{Generalized {K}ac-{M}oody algebras, automorphic forms and
  {C}onway's group. {II}}, J. Reine Angew. Math. \textbf{625} (2008), 125--154.

\bibitem[Sch11]{Sch11}
\bysame, \emph{Some constructions of modular forms for the {W}eil
  representation of {$\SL{2}(\ZZ)$}}, Preprint, 2011.

\bibitem[Sie51]{Si51}
C.~Siegel, \emph{Indefinite quadratische {F}ormen und {F}unktionentheorie.
  {I}}, Math. Ann. \textbf{124} (1951), 17--54.

\bibitem[Sko84]{Sk84}
N.~Skoruppa, \emph{{Ü}ber den {Z}usammenhang zwischen {J}acobiformen und
  {M}odulformen halbganzen {G}ewichts}, Ph.D. thesis, University of Bonn,
  Germany, 1984.

\bibitem[Sko90]{Sk90}
\bysame, \emph{Binary quadratic forms and the {F}ourier coefficients of
  elliptic and {J}acobi modular forms}, J. Reine Angew. Math. \textbf{411}
  (1990), 66--95.

\bibitem[Sko08]{Sk08}
\bysame, \emph{Jacobi forms of critical weight and {W}eil representations},
  Modular Forms on Schiermonnikoog (B.~Edixhoven, G.~v.~d. Geer, and B.~Moonen,
  eds.), Cambridge University Press, Cambridge, 2008.

\bibitem[SZ88]{SZ88}
N.~Skoruppa and D.~Zagier, \emph{Jacobi forms and a certain space of modular
  forms}, Invent. Math. \textbf{94} (1988), no.~1, 113--146.

\bibitem[SZ89]{SZ89}
\bysame, \emph{A trace formula for {J}acobi forms}, J. Reine Angew. Math.
  \textbf{393} (1989), 168--198.

\bibitem[Wal63]{Wa63}
C.~Wall, \emph{Quadratic forms on finite groups, and related topics}, Topology
  \textbf{2} (1963), 281--298.

\bibitem[Zie89]{Zi89}
C.~Ziegler, \emph{Jacobi forms of higher degree}, Abh. Math. Sem. Univ. Hamburg
  \textbf{59} (1989), 191--224.

\end{thebibliography}
